%% file: main_R4.tex
\documentclass{article}
\usepackage[utf8]{inputenc}
\usepackage{amsmath,amssymb}
\usepackage{algorithmic}
\usepackage[ruled,commentsnumbered,linesnumbered,vlined]{algorithm2e}
\usepackage{xcolor}
\usepackage{indentfirst}
\usepackage{comment}
\usepackage{longtable}
\usepackage{lscape}
\usepackage{rotating}

\usepackage{pdflscape}

\DeclareMathOperator*{\argmin}{argmin}
\DeclareMathOperator*{\argmax}{argmax}

\usepackage{lscape}
\usepackage{longtable}
\usepackage{amsthm}
\usepackage{subfigure}
\usepackage{picture}
\usepackage{graphicx}
\usepackage{fullpage}
\usepackage{hyperref}
\usepackage{float}

\usepackage{apacite}

\newtheorem{prop}{Proposition}
\newtheorem{coro}{Corollary}
\newtheorem{obs}{\textbf{Observation}}
\newcommand{\round}[1]{\ensuremath{\left\lfloor#1\right\rceil}}

\DeclareRobustCommand{\rchi}{{\mathpalette\irchi\relax}}
\newcommand{\irchi}[2]{\raisebox{\depth}{$#1\chi$}} 

\newcommand{\rafaelC}[1]{\textcolor{black}{#1}}

\title{A matheuristic approach for the $b$-coloring problem using integer programming and a multi-start multi-greedy randomized metaheuristic}
\author{ Rafael A. Melo {\thanks{Universidade Federal da Bahia, Departamento de Ci\^{e}ncia da Computa\c{c}\~{a}o, Computational Intelligence and Optimization Research Lab (CInO), Salvador, Brazil.  ({\tt melo@dcc.ufba.br}). Corresponding author.}}
	\and Michell F. Queiroz {\thanks{Universidade Federal da Bahia, Departamento de Ci\^{e}ncia da Computa\c{c}\~{a}o, Computational Intelligence and Optimization Research Lab (CInO), Salvador, Brazil.  ({\tt michellfelippe@dcc.ufba.br}).}}
	\and Marcio C. Santos {\thanks{Universidade Federal do Ceará, Campus Russas. Rua Felipe Santiago, 411. Russas, CE 62900-000. Brazil. ({\tt marciocs@ufc.br}).}} 
       }

\begin{document}

\maketitle

\begin{abstract}
    Given a graph $G=(V,E)$, the $b$-coloring problem consists in attributing a color to every vertex in $V$ such that adjacent vertices receive different colors, every color has a $b$-vertex, and the number of colors is maximized. A $b$-vertex is a vertex adjacent to vertices colored with all used colors but its own. The $b$-coloring problem is known to be NP-Hard and its optimal solution determines the $b$-chromatic number of $G$, denoted $\rchi_b(G)$. This paper presents an integer programming formulation and a very effective multi-greedy randomized heuristic which can be used in a multi-start metaheuristic. In addition, a matheuristic approach is proposed combining the multi-start multi-greedy randomized metaheuristic with a MIP (mixed integer programming) based local search procedure using the integer programming formulation. Computational experiments establish the proposed multi-start metaheuristic as very effective in generating high quality solutions, along with the matheuristic approach successfully improving several of those results. Moreover, the computational results show that the multi-start metaheuristic outperforms a state-of-the-art hybrid evolutionary metaheuristic for a subset of the large instances which were previously considered in the literature. An additional contribution of this work is the proposal of a benchmark instance set, which consists of newly generated instances as well as others available in the literature for classical graph problems, with the aim of standardizing computational comparisons of approaches for the $b$-coloring problem in future works. \\
    
\end{abstract}

\input{01_introduction.tex}

\input{02_model.tex}

\input{03_heuristic.tex}

\input{04_multistartgreedy.tex}

\input{05_experiments.tex}

\input{06_finalremarks.tex}

\vspace{0.8cm}

{
\noindent \small 
\textbf{Acknowledgments:} Work of Rafael A. Melo was supported by Universidade Federal da Bahia; the Brazilian Ministry of Science, Technology, Innovation and Communication (MCTIC); the State of Bahia Research Foundation (FAPESB); and the Brazilian National Council for Scientific and Technological Development (CNPq). Work of Michell F. Queiroz was partially supported by a CAPES scholarship.  Work of Marcio C. Santos was supported by Universidade Federal do Ceará. {The authors would like to thank the editor and the anonymous reviewers for the valuable comments which helped to improve the quality of this paper.}

}

\bibliographystyle{apacite}
\bibliography{main_R4}
 
\end{document}

%% file: 01_introduction.tex
\section{Introduction}
\label{sec:introduction}

\subsection{Basic notation and problem definition}
Given a simple graph $G=(V,E)$ and a set of colors $K = \{1,\ldots,|K|\}$, define a \textit{coloring} $c: V \rightarrow K$ as a function which assigns to each vertex $v \in G$ a color $i \in K$. A coloring is said to be \textit{proper} if $c(u) \not= c(v)$ for every $uv \in E$. {An example of proper coloring is illustrated in Figure~\ref{fig:propercoloring}.}
\begin{figure}[ht]
\centering
\includegraphics[scale=0.30]{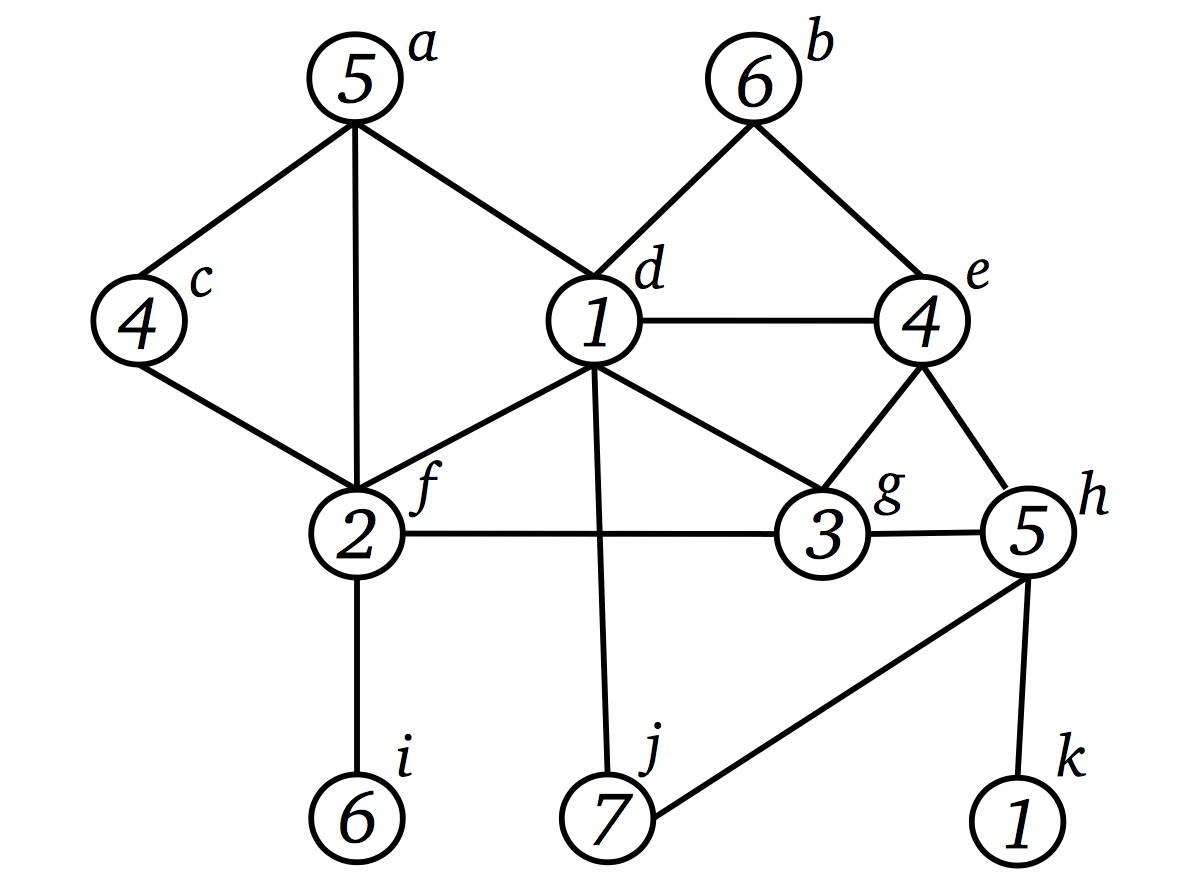}
\caption{Illustration of a proper coloring with seven colors.}
\label{fig:propercoloring}
\end{figure}

Given a coloring $c$, define $v$ to be a $b$-vertex if $v$ has at least one neighbor with each color in $K \setminus \{c(v)\}$, more precisely, $N(v) \cap \{u \in V \ | \ c(u) = i\} \not= \varnothing$ for every $i \in K\setminus \{c(v)\}$. A coloring is said to be a $b$-coloring if every color in $K$ has at least one associated $b$-vertex. Examples of $b$-colorings are illustrated in Figure~\ref{fig:bcoloring}.
Alternatively, define color classes of $c$ as the parts of a partition of $V$ into independent sets $C_{i} = \{v \in V \ | \ c(v) = i\}$ for each $i \in K$. A vertex $v \in V$ with $c(v) = j$ is called a \textit{b-vertex} for color $j$ if $v$ has a neighbor representing every other color class, i.e., $N(v) \cap C_{i} \neq \varnothing$ for all $i \in K \setminus \{j\}$. In this alternative definition, a $b$-coloring is a proper coloring such that every color class has a \textit{b-vertex}.
\begin{figure}[ht]
\centering
\subfigure[$b$-coloring with four colors in which the $b$-vertices are $d$ (color 1), $f$ (color 2), $h$ (color 2), $e$ (color 3), $a$ (color 4), and $g$ (color 4).]{
  \includegraphics[scale=0.30] {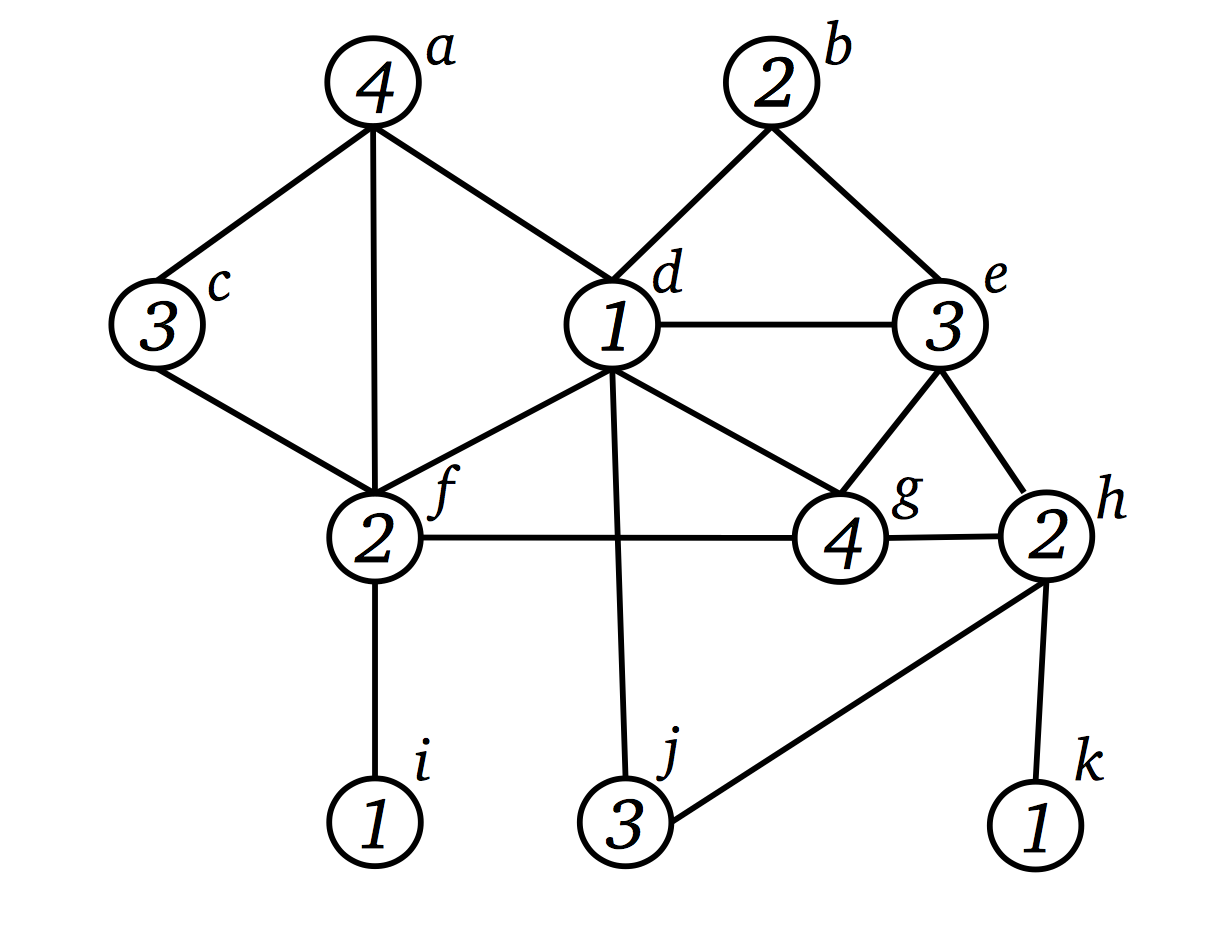}
  \label{fig:bcoloring4}
}
  \hspace{0.5cm}
  \subfigure[$b$-coloring with five colors in which the $b$-vertices are $d$ (color 1), $f$ (color 2), $g$ (color 3), $e$ (color 4), and $h$ (color 5).]{
   \includegraphics[scale=0.30] {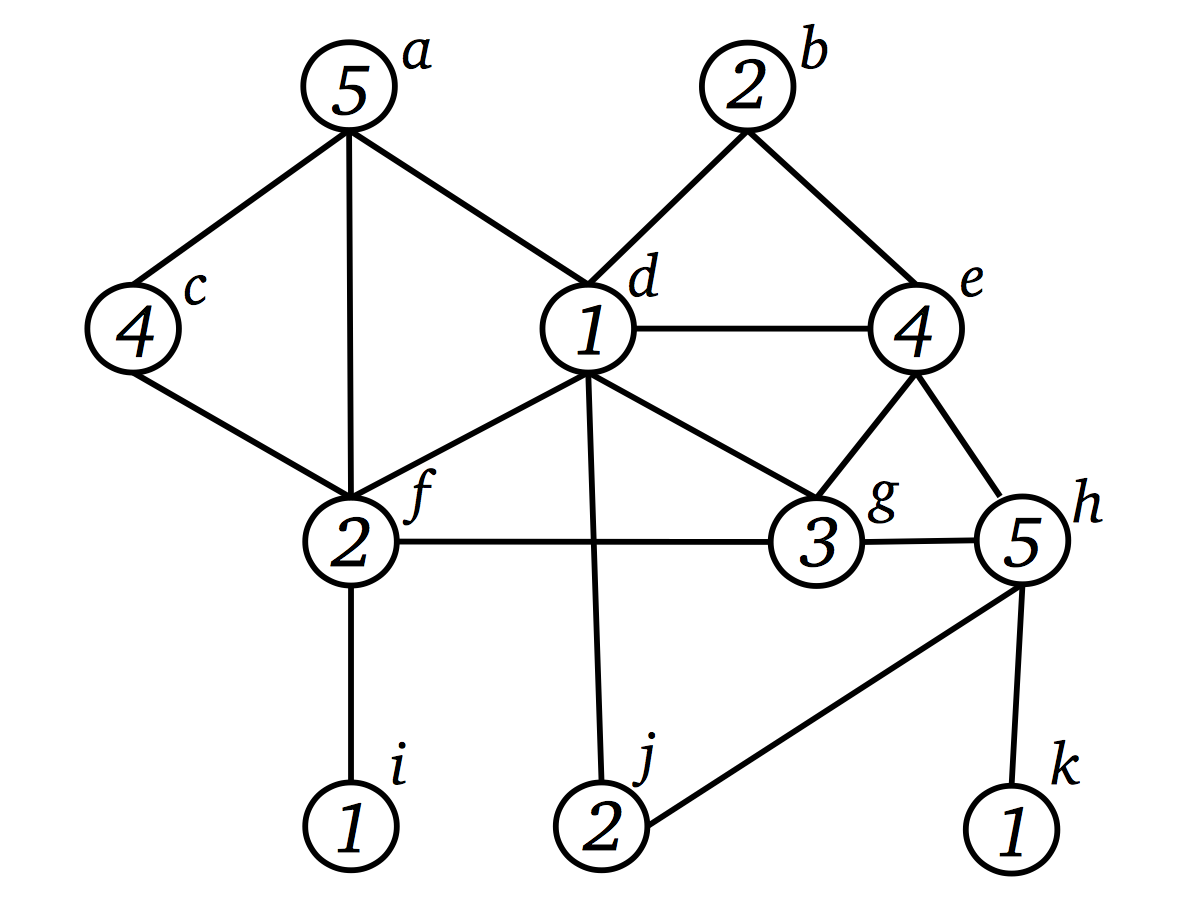}
   \label{fig:bcoloring5}
 }
\caption{Illustrations of $b$-colorings with four and five colors.}
\label{fig:bcoloring}
\end{figure}

The chromatic number of a graph $G$, $\rchi(G)$, is the minimum number of colors needed to properly color $G$. The $b$-chromatic number of a graph $G$, $\rchi_{b}(G)$, is the maximum number of colors for which $G$ admits a $b$-coloring. The coloring problem consists in encountering a proper coloring of a graph minimizing the number of colors. The $b$-coloring problem consists in encountering a proper $b$-coloring of a graph maximizing the number of colors. The problem of finding $\rchi_{b}(G)$ was shown to be NP-hard in \citeA{IrvMan99}, thus the $b$-coloring problem is NP-hard.

Although the $b$-coloring and the coloring problems appear to be closely related, they have several differences. First of all, they consider the objective functions in opposite directions and the difference between their optimal solution values can be arbitrarily large~\cite{KraTuzVoi02}. Furthermore, the $b$-coloring problem can be largely influenced by the girth (length of a shortest cycle) of the graph, what is not exactly the case for the coloring problem~\cite{girth}. Besides, a property that is commonly exploited by constructive and enumerative methods for the coloring problem is the fact that one can have solutions with the number of colors ranging from the chromatic number to the cardinality of the vertex set. However, it is not true that one can construct a $b$-coloring with $k$ colors for every integer $k$ ranging from the chromatic number to the $b$-chromatic number~\cite{Barth07}. Additionally, notice that a proper graph coloring which is not a $b$-coloring can be trivially improved by the removal of a color, namely, one that does not have a $b$-vertex. Therefore, when one is trying to minimize the number of colors, $b$-colorings appear naturally as otherwise the available coloring could be easily improved. On the other hand, when one is trying to maximize the number of colors, it is a challenging task to increase the number of colors while ensuring that $b$-vertices are generated for every new color. This suggests that the search for good quality solutions for the $b$-coloring problem should explore the structure of feasible solutions in a different manner. 

\citeA{CaLiMa15} presented a motivation for solving the $b$-coloring problem, namely, finding an upper bound for the $b$-algorithm which is a heuristic approach for the coloring problem. The $b$-algorithm works as follows, it begins with a greedy coloring and afterwards tries to reduce the number of used colors by changing the colors of certain vertices.
In this context, a $b$-vertex represents a vertex that cannot have its color changed and thus forbids further improvements by the $b$-algorithm. Hence, the $b$-chromatic number represents the worst case of the $b$-algorithm. 

Let $N(v) = \{u \in V \ | \ uv \in E \}$ be the \textit{open neighborhood} (or simply neighborhood) of $v$ in $G$, $N[v] = N(v) \cup \{v\}$ be the \textit{closed neighborhood} of $v$ in $G$, $\bar{N}(v) = V \setminus N[v]$ be the \textit{anti-neighborhood} of $v$ in $G$, and $\bar{N}[v] = \bar{N}(v) \cup \{v\}$ be the \textit{closed anti-neighborhood} of $v$ in $G$.
Define $N_{c}(v) = \{i \in K\ | \ c(u) = i \textrm{ for some } u\in N(v)  \} $ to be the set of colors adjacent to $v$, which we denote the \textit{color neighborhood} of $v$. Also, let $N_{c}[v] = N_{c}(v) \  \cup \  \{ c(v) \}$ be the \textit{color closed neighborhood}, and $\bar{N}_{c}(v) = K \setminus N_{c}[v]$ be the \textit{color anti-neighborhood} of $v$.
Denote the \textit{degree} of $v$ by $d(v)$, which is the size of its neighborhood $|N(v)|$. Considering $\Delta(G)$ to be the \textit{maximum degree} of a graph, we write $\Delta$ whenever $G$ is clear from the context. The neighborhood of a $b$-vertex can contain at most $\Delta$ colors.  Define the \textit{color degree} of $v$ by $d_c(v)$, which is the size of its color neighborhood $|N_{c}(v)|$. Consider a sorting of the vertices $V = \{v_1, v_2,\ldots,v_n\}$ such that $d(v_1) \geq d(v_2) \geq ... \geq d(v_n)$. The invariant $m(G) = max\{i\ | \ i - 1 \leq d(v_i)\}$ provides an upper bound for the $b$-chromatic number of $G$. 
Let $V_{m} \subseteq V$ be the subset of vertices with degree at least $m(G) - 1$, i.e. for each $v \in V_{m}$, $d(v) \geq m(G) - 1$. 
Denote $K_{m} \subseteq K$ as the set of colors that were attributed to some vertex in $V_m$ in a given coloring $c$, i.e. $k \in K_{m}$ if there is a vertex $v \in V_{m}$ such that $c(v) = k$.

\subsection{Literature review}

The concept of $b$-coloring appeared in different applications. \citeA{GacEglLebEmp08,GacEglLebEmp09} applied $b$-coloring to improve postal mail sorting systems, which are based on efficient optical recognition of the addresses on envelopes. The authors presented a new approach for address block localization, which is a very important step on the recognition of the addresses. Their approach uses $b$-coloring to train a classifier in the identification of the address block, and according to the authors a rate of 98\% good locations on a set of 750 envelope images was obtained.
\citeA{ElgDesHacDusKhe06} proposed a new clustering approach based on $b$-coloring of graphs. The presented cluster validation algorithm evaluates the quality of clusters based on the $b$-vertex property. The authors take on this clustering technique to detect a new typology of hospital stays in the French healthcare system.

 Several authors studied properties of $b$-coloring for special classes of graphs. \citeA{KraTuzVoi02} have shown that deciding the $b$-chromatic number is NP-Complete even for bipartite graphs. A graph $G$ is $m$-tight if it has exactly $m(G)$ vertices with degree exactly equal to $m(G) - 1$. In this regard, \citeA{HavSalSam12} proved that deciding if $\rchi_{b}(G) = m(G)$  is NP-Complete for tight chordal graphs, while showing that the $b$-chromatic number of a split graph can be obtained in polynomial time.

 Primal bound results were introduced by \citeA{IrvMan99}. We can assume that the chromatic number $\rchi(G)$ is a lower bound, as every $b$-coloring is also a proper coloring.  The upper bound is $\Delta + 1$, on account of the additional color being the color of a $b$-vertex itself. This upper bound can be narrowed, since for a $b$-coloring we need a sufficient amount of vertices of high degree. Naturally, for a $b$-coloring with $k$ colors, at least $k$ vertices with $k-1$ minimum degree are necessary. 
 As a consequence, $m(G)$ is a reduced upper bound for the problem.
 A variety of bounds on the $b$-chromatic number were also presented in~\citeA{AlkKoh11, BalRaj13, KouMah02}.

Regular graphs belong to a special class of graphs such that $m(G) = \Delta + 1$, one of the main reasons why they attract significant study. \citeA{KraTuzVoi02} have shown that for every $d$-regular graph with at least $d^4$ vertices $\rchi_{b}(G) = \Delta + 1$,  establishing that there is only a limited number of $d$-regular graphs for which $\rchi_{b}(G) < \Delta + 1$. Later, \citeA{CabJak11} proved that for every $d$-regular graph with at least $2d^3$ vertices $\rchi_{b}(G) = \Delta + 1$.  A detailed review of the literature related to the $b$-chromatic number can be found in \citeA{JakPet18}.

The $b$-coloring problem for more general graphs was considered in several works. \citeA{CorValVer05} introduced an approximation approach for the $b$-chromatic number. They have shown that the $b$-chromatic number cannot be approximated within a factor of $120/133 - \epsilon$ for any constant $\epsilon > 0$, unless P = NP. \citeA{GalKat13} settled negatively the question about the existence of a constant-factor approximation algorithm for the $b$-chromatic number, proving that for graphs with $n$ vertices, there is no $\epsilon > 0$, for which the problem can be approximated within a factor $n^{1/4-\epsilon}$, unless P = NP.
 
Despite the fact that the $b$-coloring problem has received a lot of attention from the graph theory community, just a few authors considered optimization approaches such as metaheuristics or integer programming. To the best of our knowledge, \citeA{FisPetMerCre15} were the first authors to propose a metaheuristic algorithm for the $b$-coloring problem. They proposed an hybrid evolutionary algorithm and tested its performance on a set of small instances composed of $d$-regular graphs. For the tested $d$-regular instances, the metaheuristic obtained the optimal solutions, which were attested using a brute force method. Encouraged by those results, the authors also considered larger benchmark instances from the second DIMACS implementation challenge~\cite{JoDaTr96}. As far as our knowledge goes, the only metaheuristic for the $b$-coloring problem is the one presented in~\citeA{FisPetMerCre15}, contrasting with the classical graph coloring problem, as the latter has a diversity of heuristic methods proposed in the literature~\cite{AvHeZu03, BlZu08, We90, LuHa10, MaHaMa09}.

\citeA{KocPet15} introduced an integer linear programming formulation for the $b$-chromatic index $\rchi_{b}'(G)$, the edge version of the problem. The authors also provide bounds and general results for a diversity of direct products of graphs regarding the $b$-chromatic index. \citeA{KocMar18} proposed an integer programming approach for the decision version of the $b$-coloring problem, which consists in determining whether a graph $G$ admits a $b$-coloring with a given number of colors. The authors also performed a polyhedral study of the proposed formulation, presented valid inequalities and implemented a branch-and-cut algorithm. Computational experiments were performed testing whether $\rchi_{b}(G) = m(G)$ for the input graphs.

\subsection{Main contributions and organization}

The main contributions of this paper are an integer programming formulation for the $b$-coloring problem, a very effective multi-start multi-greedy randomized metaheuristic which attempts to explore the problem structure in the search for good quality solutions, and a matheuristic approach obtained by combining the proposed multi-start metaheuristic with a fix-and-optimize local search based on the introduced integer programming formulation. To the best of our knowledge, this paper presents the first matheuristic for the $b$-coloring problem, and the first integer programming formulation which can be directly applied to its optimization version. Furthermore, we present a benchmark set consisting of newly created instances as well as available ones for coloring and maximum clique problems. 
The computational experiments show that the newly proposed approaches are very effective, reaching and proving optimality for several of the tested instances. Furthermore, the approaches are able to outperform a state-of-the-art metaheuristic~\cite{FisPetMerCre15} for the $b$-coloring problem when taking into consideration all nine large instances considered by those authors.

The remainder of the paper is organized as follows. Section~\ref{sec:formulation} introduces an integer programming formulation for the $b$-coloring problem. Section~\ref{sec:heuristic} describes the multi-greedy randomized heuristic. Section~\ref{sec:multistart} presents the multi-start multi-greedy randomized metaheuristic, the MIP (mixed-integer programming) based fix-and-optimize local search procedure using the proposed integer programming formulation, and the matheuristic approach which is obtained by combining the first two.
Section~\ref{sec:experiments} summarizes the computational experiments. Final considerations are discussed in Section~\ref{sec:finalremarks}.

%% file: 02_model.tex
\section{Integer programming formulation}
\label{sec:formulation}

We now describe a formulation by representatives \cite{CamCorFro04} for the $b$-coloring problem. Consider the binary variable $x_{uv}$ to be equal to one if vertex $u$ represents the color of vertex $v$ and to be zero otherwise, defined for every ordered pair $(u,v)$, with $u \in V$ and $v \in \bar{N}[u]$. In the proposed formulation, a vertex $u\in V$ can only represent the color of another vertex if $x_{uu}=1$, which means that $u$ is the representative and also a $b$-vertex of that color. Note that a color may have several $b$-vertices, but only one of them will be the representative.

Define the set of vertices in the anti-neighborhood of $u$ which are not adjacent to other vertices in this anti-neighborhood as $\bar{N}^{*}(u) = \bar{N}(u) - \{v \ | \ \exists w \in \bar{N}(u), vw \in E\}$. Additionally, consider the complement of $E$ as $\bar{E} = \{ uv \ | \ uv \notin E\}$. The $b$-coloring problem can be cast as the following linear integer program: 
\begin{flalign}
\ z = \  \max \ \ \ & \sum_{u \in V} x_{uu} & \label{objaux} \\
 \ \ & \sum_{v \in \bar{N}[u]} x_{vu} = 1, & \forall \ u \in V, \label{aux2} \\
& x_{uv} + x_{uw} \leq x_{uu}, & \forall \ u \in V,\ v,w \in \bar{N}(u) \textrm{ s.t. } vw \in E, \label{aux3} \\
& x_{uv} \leq x_{uu}, & \forall \ u \in V, \ v \in \bar{N}^{*}(u), \label{aux4} \\
& \sum_{w \in N(v) \cap \bar{N}(u)} x_{uw} \geq x_{uu} + x_{vv} - 1, & \forall \ (u,v) \textrm{ s.t. } uv \in \bar{E}, \label{aux5}\\
& x_{uv} \in \{0,1 \} , & \forall \  u \in V, \ v \in \bar{N}[u]. \label{aux6}
\end{flalign}

The objective function~(\ref{objaux}) maximizes the number of representative vertices, which are the $b$-vertices.
Constraints~(\ref{aux2}) ensure that every vertex must have a color. Constraints~(\ref{aux3}) force the coloring to be proper. Constraints~(\ref{aux4}) guarantee that a vertex can only give a color if it is a representative (notice that if this constraint is removed, a vertex that has a stable set as anti-neighborhood is allowed to represent all its neighborhood without being a representative). Constraints~(\ref{aux5}) are the $b$-coloring restrictions which imply that if both $u$ and $v$ are $b$-vertices, then there must be a neighbor of $v$ which is represented by $u$. 
This is achieved due to the fact that if both $u$ and $v$ are representatives, the right-hand side is equal to one, implying that the summation in the left-hand side, which is composed by the neighbors of $v$ that can be represented by $u$, should be at least one.
Constraints~(\ref{aux6}) ensure the integrality requirements on the variables.

\begin{obs}\label{obs:fixzero}
Let $\underline{z}$ be any valid lower bound for the optimal value $z^*$ of formulation (\ref{objaux})-(\ref{aux6}), i.e. $\underline{z} \leq z^*$. Let $V' \subset V$ be the set of vertices with degree strictly smaller than $\lceil \underline{z} \rceil - 1$, i.e., $d(u) < \lceil \underline{z} \rceil - 1$ for every  $u \in V'$.
Therefore, one can set to zero variables $x_{uu}$ corresponding to vertices $u \in V'$ without losing optimality, as the vertices in $V'$ can never be $b$-vertices in a $b$-coloring with at least $\lceil \underline{z} \rceil$ colors. Furthermore, variables $x_{uv}$ would also be set to zero for every pair $(u,v)$ such that $u\in V'$ and $v \in \bar{N}(u)$. 
\end{obs}

\begin{obs}\label{obs:fixzero2}
Let $\hat{x}$ be an integer feasible solution for (\ref{objaux})-(\ref{aux6}) with objective value $\hat{z}$. In any solution which strictly improves $\hat{x}$, every vertex $v \in V$ which is determined to be a representative must have degree $d(v)$ at least $\hat{z}$, i.e. $d(v) \geq \hat{z}$.
Let $V' \subset V$ be the set of vertices with degree strictly smaller than $\hat{z}$, i.e., $d(u) < \hat{z}$ for every  $u \in V'$.
Therefore, in order to obtain a solution which strictly improves $\hat{x}$, one can set to zero variables $x_{uu}$ corresponding to vertices $u \in V'$ without losing optimality in case such improving solution exists. Similarly to Observation~\ref{obs:fixzero}, variables $x_{uv}$ would also be set to zero for every pair $(u,v)$ such that $u\in V'$ and $v \in \bar{N}(u)$.   
\end{obs}

%% file: 03_heuristic.tex
\section{Multi-greedy randomized heuristic}
\label{sec:heuristic}

In this section, we present a multi-greedy randomized constructive heuristic for the $b$-coloring problem.
The heuristic follows a two-phase framework similar to the one of \citeA{ElgDesHacDusKhe06}. In the first phase, an initial proper coloring, not necessarily a $b$-coloring, is generated. The second phase ensures a proper $b$-coloring is obtained starting from the coloring achieved in the first phase. In the remainder of this section, after presenting the pseudo-code of the two-phase framework, we describe the details of the first phase in Subsection~\ref{sec:firstphase} and of the second phase in Subsection~\ref{sec:secondphase}.

The multi-greedy randomized constructive heuristic runs in $O(|V|^2 + |V|\Delta^2 \log \Delta)$ (as it will be shown in Corollary~\ref{coro:runningtime}), and is described in Algorithm~\ref{alg:const}. {It takes as inputs the graph $G = (V,E)$ and two parameters regarding the sizes of restricted candidate lists (RCL) which will be defined later in this section, namely $\alpha$ and $\beta$. The algorithm returns a proper coloring $c$ and a set of colors $K_c$.} The heuristic uses the following structures:
 \begin{itemize}
    \item $c$: structure that represents the coloring which assigns a color to each vertex $v \in V$;
    \item $N_{c}$: structure that represents the color neighborhoods of vertices $v \in V$ in coloring $c$;
    \item $K_c$: set of colors used in coloring $c$;
    \item $K_{b}$: set of colors in $K_c$ that have $b$-vertices.
 \end{itemize}
 The first phase of the approach is invoked in procedure INITIAL-COLORING (line \ref{RC-5}), which will be detailed in Section~\ref{sec:firstphase}, to obtain an initial proper coloring employing $\Delta+1$ available colors. Observe that the upper bound $\Delta+1$ was used instead of $m(G)$ with the intention of not being too restrictive and give more flexibility for the heuristic to use colors that will be removed later in the second phase of the framework. {The structures $c$, $N_{c}$, $K_{c}$, and $K_{b}$ are determined by this call to INITIAL-COLORING.}

  As it was already mentioned, procedure INITIAL-COLORING does not ensure a $b$-coloring, as some colors in $K_{c}$ might not have a $b$-vertex.
  In order to obtain a feasible $b$-coloring, the second phase is invoked in procedure FIND-B-COLORING (line \ref{RC-7}), which will be detailed in Section~\ref{sec:secondphase}, in order to remove colors from $K_c$ until a $b$-coloring is achieved. {The updated structures $c$ and $K_{c}$ are returned at the end of the execution of FIND-B-COLORING.} RANDOMIZED-CONSTRUCTIVE thus returns the obtained $b$-coloring $c$ as well as the {set of used colors $K_c$} (line \ref{RC-8}). 
  
\begin{algorithm}[H]
\caption {{RANDOMIZED-CONSTRUCTIVE($G, \alpha, \beta$)}}
\label{alg:const}

    {$c, N_{c}, K_c, K_{b} \leftarrow$ INITIAL-COLORING($G, \alpha$)}\;\label{RC-5}
    {$c, K_c\leftarrow$} FIND-B-COLORING($G,\alpha, \beta, N_{c}, K_{b}, c, K_c$)\;\label{RC-7}

    \Return $b$-coloring $c$, $K_c$\;\label{RC-8}
\end{algorithm}

\subsection{First phase: obtaining an initial coloring}
\label{sec:firstphase}

An initial coloring is obtained using procedure INITIAL-COLORING, which is detailed in Algorithm~\ref{alg:initb}. {In addition to the graph $G$, the algorithm also takes as input a parameter $\alpha$ related to the size of restricted candidate lists. The structures $c$, $N_{c}$, $K_c$, and $K_{b}$ will be returned at the end of its execution.}
We remark that, for ease of explanation, the pseudo-code which will be presented assumes the graph is connected. An easy way to overcome this fact will be given once the algorithm is described.
The following structures are used by the algorithm:
\begin{itemize}
    \item $K'$: initial set of available colors;
    \item $Q$: stores the set of vertices which had already been colored;
    \item $\Upsilon_{v}$: keeps the vertices in $N(v)$ which have no attributed color;
    \item {$K_{m}$: set of colors in $K_c$ that were attributed to some vertex $v \in V$ with degree at least $m(G)-1$.}
\end{itemize}
INITIAL-COLORING uses the following auxiliary {method}:
\begin{itemize}
    \item {HEURISTIC-COLOR-VERTEX}: described in Algorithm \ref{alg:coloru}, the procedure takes as inputs the graph $G$, vertices $v$ and $u$, {as well as structures $N_{c}$, $K'$, $K_{m}$. The method returns a color to be attributed to vertex $u$. Firstly}, structure $LC$ is initialized as empty (line \ref{INIT-5a}), and will store the set of candidate colors for coloring $u$. The algorithm then checks if $u$ has degree greater than or equal to $m(G) - 1$ (line \ref{INIT-5b}) and tries initially to build $LC$ with the set of colors $k \in K'$ not belonging to the color neighborhood of neither $v$ nor $u$, and have not been assigned to a vertex with degree greater than or equal to $m(G) - 1$, i.e., $k \not\in N_{c}(v)$, $k \not\in N_{c}(u)$ and $k \not\in K_{m}$ (line \ref{INIT-6}). The purpose behind this coloring idea is to diversify the colors assigned to both neighborhoods of $u$ and $v$, while trying to give different colors to vertices with high enough degrees to become $b$-vertices, in an attempt to increase the probability of finding $b$-vertices that represent the greater amount of color classes. If $LC$ is still empty (line \ref{INIT-6a}), the algorithm tries to include in $LC$ colors $k \in K'$ not belonging to the color neighborhood of neither $v$ nor $u$, i.e.,  $k \not\in N_{c}(v)$ and $k \not\in N_{c}(u)$ (line \ref{INIT-6b}). If no such color exists, i.e., $LC$ remains with no elements in line \ref{INIT-7}, $LC$ is built in line \ref{INIT-8} with colors $k \in K'$ not belonging to the color neighborhood of $u$, i.e.,  $k \not\in N_{c}(u)$. This guarantees at least one color in $LC$ since the algorithm initially works with $\Delta+1$ available colors and $d(u) \leq \Delta$. 
    {The color in $LC$ with lowest index is returned in line \ref{INIT-8rr}.}
\end{itemize}

 \begin{algorithm}[H]
\caption {{INITIAL-COLORING($G, \alpha$)}}
\label{alg:initb}
    {$c(v) \leftarrow 0$ and $N_{c}(v) \leftarrow \varnothing$ for each $v \in V$ \;\label{RC-4a}
     $K_c, K_{m}, K_{b} \leftarrow  \varnothing$\;\label{RC-4}}
     
      $v_\Delta \leftarrow \argmax_{u \in V} \{ |N(u)|  \}$\;\label{INIT-0a}
     $K' \leftarrow \{1,2,3,...,\Delta+1\}$\;\label{RC-1}
     $c(v_{\Delta}) \leftarrow 1$\;\label{RC-2}
      {Update $N_{c}(v')$ for each $v'\in N(v_{\Delta})$, $K_c$ and $K_{m}$} \;\label{RC-2a}
     $Q \leftarrow \{v_{\Delta}\}$\; \label{INIT-0}
     \While {$Q \neq \varnothing$} { \label{INIT-1}
        Create {RCL$_{Q}(\alpha)$} with the best elements in $Q$\;\label{INIT-2}
        $v \leftarrow$ vertex randomly selected from {RCL$_{Q}(\alpha)$} \;\label{INIT-2a}
         $\Upsilon_{v} \leftarrow \{ w \in N(v) \ | \ c(w) = 0 \}$\;\label{INIT-3}
         
         \While {$\Upsilon_{v} \neq \varnothing$} { \label{INIT-4}
         Create {RCL$_{\Upsilon_{v}}(\alpha)$} with the best elements in $\Upsilon_{v}$\;\label{INIT-4a}
        $u \leftarrow$ vertex randomly selected from {RCL$_{\Upsilon_{v}}(\alpha)$} \;\label{INIT-4b}
            
                 {$c(u) \leftarrow$ {HEURISTIC-COLOR-VERTEX}($G,v,u,K',K_{m}$)}\;\label{INIT-8a}
                 {Update, if necessary, $N_{c}(v')$ for each $v'\in N(u)$, $K_c$ and $K_{m}$} \;\label{INIT-8b}
                 $Q \leftarrow Q \cup \{u\}$\;\label{INIT-10}
                 $\Upsilon_{v} \leftarrow \Upsilon_{v} \setminus \{u\}$\;\label{INIT-11}

         }
         $Q \leftarrow Q \setminus \{v\}$\;\label{INIT-12}
    }

         {$K_{b} \leftarrow \{c(v) \ | \ v\in V \textrm{ and }N_{c}[v]=K_{c} \}$}\;\label{INIT-13}
    \Return $c$, $N_{c}, K_c, K_{b}$\; \label{INIT-18}
\end{algorithm}

 {Algorithm~\ref{alg:initb} first initializes the used structures as follows. For each vertex $v \in V$, the color neighborhood of $v$ is initialized as empty and $c(v)$ is set to 0, which implies that no color is assigned to $v$ (line \ref{RC-4a}). The sets $K_{c}$, $K_{m}$ and $K_{b}$ are initialized as empty (line \ref{RC-4}). Next, the algorithm sets} $v_{\Delta}$ as the maximum degree vertex in $G$ in line \ref{INIT-0a}, where ties are broken arbitrarily. The set $K'$ is initialized with $\Delta+1$ colors in line \ref{RC-1}, followed by the coloring of $v_{\Delta}$  with color 1 in line \ref{RC-2}. 
 {The structures $N_c$, $K_c$ and $K_m$ are updated in line~\ref{RC-2a}.}
 The neighborhood of vertices in $Q$ are yet to be explored, and the set is initialized with $v_{\Delta}$ in line \ref{INIT-0}. The algorithm then performs a series of iterations to assign colors to the vertices in $G$ while the set $Q$ is not empty in lines \ref{INIT-1}-\ref{INIT-12}. Elements from a restricted candidate list (RCL) containing the best elements in $Q$ are randomly chosen along the construction of the solution. Given the vertices in $Q$ the greedy choice criterion for {RCL$_{Q}(\alpha)$} is:
\begin{itemize}
    \item maximization of the vertex degree: $p_1 = \max_{v \in Q} d(v)$.
\end{itemize}
{RCL$_{Q}(\alpha)$} is defined as a subset of $Q$ containing all candidates whose evaluation for the greedy criterion lies in an interval of values defined by a parameter $\alpha \in [0.0,1.0]$. Define $\underline{p_1} = \min_{v \in Q} d(v)$, thus this interval is given by $[p_1 - \alpha(p_1-\underline{p_1}), p_1]$. {RCL$_{Q}(\alpha)$} is created in line \ref{INIT-2}. A vertex $v$ is randomly selected from {RCL$_{Q}(\alpha)$} in line \ref{INIT-2a}. Set $\Upsilon_{v}$ is built in line \ref{INIT-3} with the vertices in $N(v)$ that have no assigned color. Similar to {RCL$_{Q}(\alpha)$}, elements from a restricted candidate list containing the best elements in $\Upsilon_{v}$ are randomly chosen along the construction of the solution. Given the vertices in $\Upsilon_{v}$, the greedy choice criterion for  {RCL$_{\Upsilon_{v}}(\alpha)$} is:
\begin{itemize}
    \item maximization of the vertex degree:  $p_2 = \max_{v \in \Upsilon_{v}} d(v)$.
\end{itemize}
{RCL$_{\Upsilon_{v}}(\alpha)$} is defined {for $\Upsilon_{v}$ as RCL$_{Q}(\alpha)$ was defined for $Q$}. {RCL$_{\Upsilon_{v}}(\alpha)$} is created in line \ref{INIT-4a}. A vertex $u$ is randomly selected from {RCL$_{\Upsilon_{v}}(\alpha)$} in line \ref{INIT-4b} and {receives a color determined by} procedure {HEURISTIC-COLOR-VERTEX} in line \ref{INIT-8a}. 
{The structures $N_c$, $K_c$ and $K_m$ are updated in line~\ref{INIT-8b}.}
The neighborhood of $u$ is yet to be explored, so the vertex is inserted into $Q$ in line \ref{INIT-10}. Vertex $u$ is then removed from $\Upsilon_{v}$ in line \ref{INIT-11} and a new iteration resumes, until set $\Upsilon_{v}$ becomes empty. 
 Vertex $v$ is then withdrawn from $Q$ in line \ref{INIT-12}. After all vertices have been colored, i.e., $Q$ is empty, the algorithm updates the list $K_{b}$ of colors having  $b$-vertices in {line \ref{INIT-13}}. The {structures $c$, $N_{c}$, $K_c$ and $K_{b}$} are then returned in line \ref{INIT-18}. Note that, for ease of explanation, the described pseudo-code assumes the graph is connected. However, in the case of a disconnected graph, this can be overcome by simply inserting into $Q$ the uncolored vertex of highest degree (if there is at least one uncolored vertex) as a last step in the loop of lines \ref{INIT-1}-\ref{INIT-12} whenever $Q$ becomes empty.

\begin{prop}\label{prop:firstphase}
    Algorithm~\ref{alg:initb} runs in $O(|V|^2)$.
\end{prop}
\begin{proof}
Consider $Q$ and $\Upsilon_v$ to be ordered lists containing vertices sorted in nonincreasing order of vertex degree, which means that every element entering these lists should be inserted into the correct ordered position. Additionally, assume $N_c(v)$ for each $v\in V$, $K_c$ and $K_m$ to be represented as $\Delta+1$-dimensional binary vectors, with each element $k$ representing whether color $k$ belongs to the corresponding set or not. 
{Firstly, consider the running time to perform a single update of the structures. Note that there are $O(\Delta)$ updates of structures $N_c(v')$ and each of them can be done in $O(1)$. The updates of $K_c$ and $K_m$ can all be done in $O(1)$. Thus, a single update of all the required structures can be done in $O(\Delta)$.}
{HEURISTIC-COLOR-VERTEX} runs in $O(\Delta)$, which is implied by the construction of $LC$ and the selection of its minimum value.
In Algorithm~\ref{alg:initb}, the instructions of lines \ref{INIT-0a}-\ref{INIT-0} run in $O(|V|)$.
{Line \ref{INIT-13} can be done in $O(|V|\Delta)$.}
In order to determine the complexity of the while loop in lines \ref{INIT-1}-\ref{INIT-12}, we perform an aggregated analysis.  Note that each vertex $v\in V$ is inserted into and removed from $Q$ at most once and each insertion into this ordered list can be performed in $O(|V|)$, implying $O(|V|^2)$ for all the insertions. As $Q$ is kept as an ordered list, whenever a vertex is to be removed from $Q$, line \ref{INIT-2} is carried out in $O(|V|)$. 
At the moment a vertex enters $Q$ lines \ref{INIT-4a}-\ref{INIT-11} are executed in $O(|\Delta|)$. We ommit the entrance of vertices in $\Upsilon_v$ from the analysis as they are directly related to their entrance in $Q$, i.e., whenever a vertex enters $\Upsilon_v$ in line~\ref{INIT-3} it will be removed from $\Upsilon_v$ in line~\ref{INIT-11} just after its entrance in $Q$.
Therefore, the overall running time of Algorithm~\ref{alg:initb} is $O(|V|^2 + |V| ( |V| + \Delta))$ which is $O(|V|^2)$.
\end{proof}

 
                 

\begin{algorithm}[H]
\caption {{{HEURISTIC-COLOR-VERTEX}($G,v,u,K',K_{m}$)}}
\label{alg:coloru}

            $LC \leftarrow \varnothing$\; \label{INIT-5a}
            \If{$d(u) \geq m(G)-1$}{\label{INIT-5b}
             $LC \leftarrow \{k \ | \ k \in K', k \not\in N_{c}(u) $, $k \not\in N_{c}(v)$ and $k \not\in K_{m}\}$\;\label{INIT-6}
             }
              \If{$LC = \varnothing$}{\label{INIT-6a}
              $LC \leftarrow \{k \ | \ k \in K', k \not\in N_{c}(u) $ and $k \not\in N_{c}(v)\}$\;\label{INIT-6b}
              }
             \If{$LC = \varnothing$}{\label{INIT-7}
              $LC \leftarrow$ $\{k \ | \ k \in K'$ and $k \not\in N_{c}(u)$\}\;\label{INIT-8} 
                 }

                   
                   \Return {$\min\{k \ | \ k \in LC\}$}\;\label{INIT-8rr}
          
\end{algorithm}


\subsection{Second phase: transforming the initial coloring into a $b$-coloring}
\label{sec:secondphase}
A feasible $b$-coloring is obtained using procedure FIND-B-COLORING, which is detailed in Algorithm \ref{alg:findb}. In addition to the graph $G$ and RCL size parameters $\alpha$ and $\beta$, the algorithm also takes as inputs the sets $N_{c}$, $K_{b}$, $K_c$, and the coloring $c$. Remark that the inputs $c$ and $K_c$ will be updated by the algorithm and will be returned at the end of its execution.
The following structure is used:
\begin{itemize}
    \item $\bar{K}_b$: set of colors that do not have $b$-vertices. 
\end{itemize}

\begin{algorithm}[H]
\caption {FIND-B-COLORING($G, \alpha, \beta, N_{c}, K_{b}, c, K_c$)}
\label{alg:findb}
     $\bar{K}_b \leftarrow K_c \setminus K_{b}$\;\label{RC-6}
    \While {$\bar{K}_b \neq \varnothing$} { \label{FC-0}
         Create {RCL$_{\bar{K}_b}(\beta)$} with the best elements in $\bar{K}_b$\; \label{FC-1}
         $r$ $\leftarrow$ color randomly selected from  {RCL$_{\bar{K}_b}(\beta)$}\; \label{FC-2}
         
         
         \ForEach{$v \in V$ such that  $c(v)=r$ }{ \label{FC-4}
             Create {RCL$_{\bar{N}_{c}(v)}(\alpha,\beta)$} with the best elements in $\bar{N}_{c}(v)$\; \label{FC-4a}
              {$c(v) \leftarrow$ color randomly selected from {RCL$_{\bar{N}_{c}(v)}(\alpha,\beta)$}}\;\label{FC-5}
             {Update, if necessary, $N_{c}(v')$ for each $v'\in N(v)$} \; \label{FC-6}
         }
         $K_c \leftarrow K_c \setminus \{r\}$\;  \label{FC-7a}
         
        \ForEach{$v \in V$ such that $c(v) \in \bar{K}_b$}{ \label{FC-7}
         {\If{$N_{c}(v) \cup \{c(v)\} =   K_c $}
         { \label{FC-9}
         $\bar{K}_b \leftarrow \bar{K}_b \setminus \{c(v)\} $\; \label{FC-11}
         }
        } }
        $\bar{K}_b \leftarrow \bar{K}_b \setminus \{r\}$\; \label{FC-12}
    }
    \Return $c, K_c$\; \label{FC-13}
\end{algorithm}

FIND-B-COLORING, which is a modification of the $b$-algorithm mentioned in the introduction, consists in iteratively eliminating colors from the graph by recoloring vertices colored with colors in $\bar{K}_b$. The set $\bar{K}_b$ is initialized with every color in {$K_c\setminus K_{b}$} in line \ref{RC-6}. The algorithm then performs a series of iterations while $\bar{K}_b$ is not empty (lines \ref{FC-0}-\ref{FC-12}). Elements from a restricted candidate list containing the best elements in $\bar{K}_b$ are randomly chosen along the construction of the solution. Given the colors in $\bar{K}_b$ the greedy choice criterion for  {RCL$_{\bar{K}_b}(\beta)$} is:
\begin{itemize}
     \item maximization of the color index: $p_3 = \max_{r \in \bar{K}_b} r$;
\end{itemize}
{Criterion $p_{3}$ aims to remove colors with higher index since after the execution of Algorithm~\ref{alg:initb}, colors with smaller index are presumably closer to have a $b$-vertex.}   {RCL$_{\bar{K}_b}(\beta)$} is defined as a subset of $\bar{K}_b$ containing its $\beta$ best candidates.  {RCL$_{\bar{K}_b}(\beta)$} is created in line \ref{FC-1}. A color $r$ is randomly selected from  {RCL$_{\bar{K}_b}(\beta)$} in line \ref{FC-2}. For each vertex $v \in V$ colored with $r$, i.e., $c(v) = r$, a new color is assigned to $v$ (lines \ref{FC-4}-\ref{FC-6}). Note that any color in $\bar{N}_{c}(v)$ is avaiable to color $v$. Elements from a restricted candidate list containing the best elements in $\bar{N}_{c}(v)$ are randomly chosen along the construction of the solution. Before explaining the greedy criterion, let $\zeta_{rv}$ be the number of vertices adjacent to $v$ such that color $r \in \bar{N}_{c}(v)$ is also not in their color neighborhood. Additionally, let $M_{uv} \subseteq \bar{N}_{c}(v)$ be the set of colors not adjacent to neither $u$ nor $v$. Define $M_{uv}^* = \argmin_{M_{uv}:u\in N(v)} |M_{uv}| $  as the minimum cardinality set among all $M_{uv}$ for $u \in N(v)$. Note that $M_{uv}^*$ is the set of colors not adjacent to the vertex with the minimum number of missing colors in its color neighborhood. Given the colors in $\bar{N}_{c}(v)$ the greedy choice criteria for {RCL$_{\bar{N}_{c}(v)}(\alpha,\beta)$} are:
\begin{itemize}
     \item {maximization of vertices with a new color added to their color neighborhood}: $p_4 = \max_{r \in \bar{N}_{c}(v)} \zeta_{rv}$;
     \item minimization of the color index considering the colors in $M_{uv}^*$: $p_5 = \min_{r \in M_{uv}^*} r$.
\end{itemize}

Criterion $p_{4}$ intends to increase the color neighborhood of as many vertices as possible, whereas $p_{5}$ aims to predict the vertex which is the closest to become a $b$-vertex. Given $p_{4}$, {RCL$_{\bar{N}_{c}(v)}(\alpha,\beta)$} is defined as a subset of $\bar{N}_{c}(v)$ containing all candidates whose evalutation of the greedy criterion lie in an interval of values defined by a parameter $\alpha \in [0.0,1.0]$. Define $\underline{p_4} = \min_{r \in \bar{N}_{c}(v)} \zeta_{rv}$, thus this interval is given by $[p_4 - \alpha(p_4-\underline{p_4}), p_4]$. As for $p_{5}$, RCL$_{\bar{N}_{c}(v)}$ is defined as a subset of ${\bar{N}_{c}(v)}$ containing its $\beta$ best candidates. 

{RCL$_{\bar{N}_{c}(v)}(\alpha,\beta)$} is created in line \ref{FC-4a}. Any of the greedy functions $p_{4}$ or $p_{5}$ can be chosen for the construction of {RCL$_{\bar{N}_{c}(v)}(\alpha,\beta)$} and they are selected at random with 50\% chance each. {Note that, as stated previously on the definition of $p_{4}$ and $p_{5}$, the selection of the one to be used will define if RCL$_{\bar{N}_{c}(v)}(\alpha,\beta)$ uses $\alpha$ or $\beta$.} 
{Vertex $v$ receives a color randomly selected from {RCL$_{\bar{N}_{c}(v)}(\alpha,\beta)$} in line~\ref{FC-5}.}
Color neighborhood of vertices in $N(v)$ are then updated in line \ref{FC-6}. After all vertices previously colored with $r$ have been assigned a new color, $r$ is removed from set $K_c$ in line \ref{FC-7a}. 

The algorithm then certifies if colors in $\bar{K}_b$ now have a $b$-vertex in lines \ref{FC-7}-\ref{FC-11}.  {Colors that now have a $b$-vertex are removed from $\bar{K}_b$ in line \ref{FC-11}}. Lastly, $r$ is removed from $\bar{K}_b$ in line \ref{FC-12}. The algorithm terminates when $\bar{K}_b = \varnothing$ which implies $K_{b} = K_c$, so the resulting $b$-coloring $c$ and the set of used colors $K_c$ are returned in line \ref{FC-13}.

\begin{prop}\label{prop:secondphase}
    Algorithm~\ref{alg:findb} runs in $O(|V|\Delta^2 \log \Delta)$.
\end{prop}
\begin{proof}
Observe that Algorithm~\ref{alg:findb} performs a series of color removals and updates. The while loop of lines \ref{FC-0}-\ref{FC-12} is executed $O(\Delta)$ times, as each color is removed at most once. On any occasion a color $r$ is to be removed from $\bar{K}_b$, line~\ref{FC-1} is carried out in $O(\Delta\log \Delta)$. The foreach loop of lines~\ref{FC-4}-\ref{FC-6} is executed $O(|V|)$ times and each iteration is performed in $O(\Delta\log \Delta)$, therefore the complete loop is executed in $O(|V|\Delta \log \Delta)$. The foreach loop of lines \ref{FC-7}-\ref{FC-11} is also executed $O(|V|)$ times and the verification and possible updates are all performed in $O(1)$ for each iteration, consequently the complete loop is executed in $O(|V|)$.
Note that in order to perform the verification of line~\ref{FC-9} in $O(1)$, one could keep for each $v\in V$ an indicator vector corresponding to $N_c(v)$ together with the number of nonzero entries in this vector, as well as an indicator vector corresponding to $K_c$ in conjunction with the number of nonzero entries in this vector. The verification could thus be performed by simply comparing the number of nonzero entries in these two indicator vectors.
Algorithm~\ref{alg:findb} thus runs in $O(\Delta ( \Delta\log \Delta + |V|\Delta \log \Delta + |V| ))$, which is $O(|V|\Delta^2 \log \Delta)$.
\end{proof}

\begin{coro}\label{coro:runningtime}
    Algorithm~\ref{alg:const} runs in $O(|V|^2 + |V|\Delta^2 \log \Delta)$.
\end{coro}
\begin{proof}
The result follows from Propositions~\ref{prop:firstphase}~and~\ref{prop:secondphase}. Note that the running time of the algorithm is dominated by the calls to Algorithms~\ref{alg:initb}~and~\ref{alg:findb}, and therefore runs in $O(|V|^2 + |V|\Delta^2 \log \Delta)$.
\end{proof}

%% file: 04_multistartgreedy.tex
\section{The matheuristic approach}
\label{sec:multistart}

In this section, before describing the matheuristic approach, we present its two main components: (a) the multi-start multi-greedy randomized metaheuristic and (b) the MIP (mixed integer programming) based fix-and-optimize local search procedure. The multi-start metaheuristic consists in performing a predefined number of iterations of the multi-greedy randomized heuristic and is described in Subsection~\ref{sub:msa}. The MIP-based fix-and-optimize local search consists in solving a restricted MIP obtained by fixing certain decision variables and is described in Subsection~\ref{sub:miph}. Finally, Subsection~\ref{sub:matheur} presents the matheuristic approach which consists in the combination of the multi-start metaheuristic with the MIP-based fix-and-optimize local search procedure.

\subsection{Multi-start multi-greedy randomized metaheuristic}\label{sub:msa}
The pseudo-code of the multi-start multi-greedy randomized metaheuristic is described in Algorithm~\ref{alg:multis}. In addition to the graph $G$ and two parameters regarding the sizes of restricted candidate lists (RCL), namely $\alpha$ and $\beta$, the algorithm also takes as input $it_{max}$, which represents the maximum number of iterations that the multi-greedy randomized heuristic will be executed.

\begin{algorithm}[H]
\caption {MULTISTART-B-COL($G, \alpha, \beta, it_{max}$)}
\label{alg:multis}

     $K_c^* \leftarrow \varnothing$ \;\label{MS-0}
    \For{$i = 1,...,it_{max} $}{\label{MS-1}
        {$c_{i}, K_{c_i} \leftarrow$ RANDOMIZED-CONSTRUCTIVE($G, \alpha, \beta$)\;\label{MS-2}}
        \If{$|K_{c_i}| > |K_c^*|$}{\label{MS-3}
            $c^* \leftarrow c_{i}$\;\label{MS-3.1}
            $K_c^* \leftarrow K_{c_i}$\;\label{MS-3.2}
            
            \If{$|K_c^*| = m(G)$}{
                \Return $c^*$, $K_c^*$\;\label{MS-3.a}
            }

        }
    }\label{MS-1end}
    \Return $c^*$, $K_c^*$\;\label{MS-4}
\end{algorithm}

Procedure MULTISTART-B-COL will save in $c^*$ the best obtained coloring. The set of used colors in $c^*$, $K_c^*$, is initialized as empty (Algorithm~\ref{alg:multis}, line~\ref{MS-0}). The coloring generated at iteration $i \leq it_{max}$ is represented by $c_{i}$ and the corresponding set of used colors as $K_{c_i}$. $|K_{c_i}|$ represents the solution value, which is the number of used colors in coloring $c_{i}$. The loop in lines~\ref{MS-1}--\ref{MS-1end} performs iterations $i=1,\ldots,\mathit{it}_\mathit{max}$. The construction phase starts by invoking procedure RANDOMIZED-CONSTRUCTIVE to build the solution $c_{i}$ in line~\ref{MS-2}. In case an improving solution is obtained, the algorithm updates $c^*$  and $K_c^*$ in lines \ref{MS-3.1}-\ref{MS-3.2}. If the solution value of $c^*$ matches the upper bound $m(G)$ the execution of RANDOMIZED-CONSTRUCTIVE is terminated by returning $c^*$ in line \ref{MS-3.a}, as the solution is proven to be optimal. Otherwise, a new iteration begins until the maximum number of iterations $it_{max}$ is exceeded. The solution with the highest number of used colors, i.e., the best solution $c^*$ encountered by the multi-start phase, is returned in line~\ref{MS-4}.

\subsection{MIP-based fix-and-optimize local search}\label{sub:miph}

Given an available feasible solution, the MIP-based fix-and-optimize local search procedure consists in generating a subproblem obtained from the original $b$-coloring problem by fixing certain decision variables at the values they assume in the available feasible solution which is also offered as a warm start for the used MIP solver. With fewer variables remaining to be optimized, it is expected that the resulting subproblem is more tractable by a standard MIP solver than the original problem. In this work, the input feasible solution consists of the best solution generated by  MULTISTART-B-COL. The MIP-based fix-and-optimize local search is described in Algorithm~\ref{alg:mipheur}. 
In addition to the graph $G$ and an initial feasible solution, represented by $c$ and $K_c$, the algorithm also takes as input the maximum time allowed for solving the obtained MIP formulation given by MAXTIME. In our framework, the initial feasible solution offered to MIP-LS will be the currently best known solution returned by MULTISTART-B-COL.

\begin{algorithm}[H]
\caption {MIP-LS($G$, $c$, $K_c$, MAXTIME)}
\label{alg:mipheur}
$V^b, V^0 \leftarrow \varnothing$\;\label{mipheur-1}

\ForEach{$k \in K_c$}{\label{mipheur-3}
$u \leftarrow \argmax_{v \in V}\{d(v) \ | \ c(v) = k,\ N_c(v) = K_c\setminus \{k\} \}$\;\label{mipheur-4}
$V^b \leftarrow V^b \cup \{u\}$ \;\label{mipheur-5}
}
\ForEach{$u \in V \setminus V^b$}{\label{mipheur-6}
\If{$d(u) < |K_c|$}{\label{mipheur-7}
$V^0 \leftarrow V^0 \cup \{u\}$\;}\label{mipheur-8}
}
Solve the MIP (\ref{objaux})-(\ref{aux6}) with addition of constraints (\ref{fixing:21})-(\ref{fixing:20}) and $c$ given as warm start, restricted to time limit MAXTIME{, in order to obtain coloring $c^*$ using colors $K_c^*$} \; \label{mipheur-9}
\Return the best solution $c^*$, $K_c^*$ encountered by the MIP\; \label{mipheur-10}
\end{algorithm}

The set of representative $b$-vertices $V^b \subseteq V$ and the set of vertices $V^0 \subseteq V$ that cannot be representatives in an improving solution are initialized as empty in line~\ref{mipheur-1} of Algorithm~\ref{alg:mipheur}.
{Set $V^b$ is built according to the input solution in the foreach loop of lines~\ref{mipheur-3}-\ref{mipheur-5}.
 Following Observations~\ref{obs:fixzero}~and~\ref{obs:fixzero2},} set $V^0$ is built from the input solution in the foreach loop of lines~\ref{mipheur-6}-\ref{mipheur-8} {with all vertices which are not $b$-vertices in coloring $c$ and have degree strictly smaller than its number of colors $|K_c|$.}
 Line~\ref{mipheur-9} solves a mixed integer program defined by the formulation presented in Section~\ref{sec:formulation}, in which all variables in $V^b$ are fixed to one (i.e., all corresponding vertices are selected to be representatives in the solution) and all vertices in $V^0$ are fixed to zero. Additionally, coloring $c$ is provided as a warm start, i.e., as an initial feasible solution.
Fixing is achieved by adding the following additional constraints to the formulation
\begin{align}
    x_{uu} = 1, \ \ \ & \forall \ u \in V^b,
    \label{fixing:21}\\
    x_{uv} = 0, \ \ \ & \forall \ u \in V^0,\ v \in \bar{N}[u]. \label{fixing:20}
\end{align}
The best solution obtained by the resulting MIP restricted to a maximum time limit MAXTIME is returned in line~\ref{mipheur-10}. {Note that the input coloring $c$ is always feasible for this MIP subproblem.}

\subsection{Matheuristic approach} \label{sub:matheur}

Combinations of metaheuristics with exact algorithms from mathematical programming approaches such as mixed integer programming (MIP), called matheuristics, have received considerable attention over the last few years. It has been acknowledged by the optimization research community that combining effort from exact and metaheuristic approaches could achieve better solutions when compared with pure classic methods \cite{RaPu00, DuSt09}. Matheuristics frequently benefit from metaheuristics as the main method to compute good quality solutions, with the exact approach used to enhance these solutions by solving subproblems.
Motivated by recently successful results by matheuristics~\rafaelC{\cite{DoNiVo18, CuKrMe19, PeLaLuRi19, MelQueRib21}}, we combine the multi-start metaheuristic MULTISTART-B-COL
that appears in Algorithm~\ref{alg:multis}
with the MIP-based fix-and-optimize local search procedure presented in Algorithm~\ref{alg:mipheur}, which produces the matheuristic MSBCOL$^+$: \\

\noindent  \textbf{MSBCOL$^+$}\\ 
\textbf{Step 1:}  $c^*$, $K_c^*$ $\leftarrow$ MULTISTART-B-COL($G, \alpha, \beta, it_{max}$);\\
\textbf{Step 2:}  $c'^{*}$, $K_c'^{*}$ $\leftarrow$ MIP-LS($G$, $c^*$, $K_c^*$, MAXTIME);\\
\textbf{Step 3:} Return $c'^*$, $K_c'^*$.\\

%% file: 05_experiments.tex
\section{Computational experiments}
\label{sec:experiments}

All computational experiments were carried out on a machine running under Ubuntu x86-64 GNU/Linux, with an Intel Core i7-8700 Hexa-Core 3.20GHz processor and 16Gb of RAM. The {meta}heuristic was coded in C++ and the formulation solved using CPLEX 12.8 under standard configurations. 
\rafaelC{Each execution of the solver was limited to one hour (3,600s).}
Subsection~\ref{sec:instancias} describes the benchmark instances. Subsection~\ref{sec:testedapproaches} lists the tested approaches and reports the parameter settings.
Subsections~\ref{sec:small} and \ref{sec:large} summarize the computational results for small and large instances, correspondingly. 
Finally, Subsection~\ref{sec:compare} compares some of the obtained computational results with a state-of-the-art metaheuristic presented in {\citeA{FisPetMerCre15}} taking into consideration a subset of the large instances.

\subsection{Benchmark instances}\label{sec:instancias}

\noindent The tests were carried out on a set of benchmark instances divided into small ($\leq$ 10,000 edges) and large ($>$ 10,000 edges) graphs, and is composed of: 
\begin{itemize}
\item[(a)] new randomly generated instances; 
\item[(b)]  instances from the Second DIMACS Implementation Challenge.
\end{itemize}

 The new set of instances was constructed using the graph generator \textit{ggen}~\cite{ggen} and includes bipartite, geometric and random graphs. 
 \rafaelC{Small i}nstances were created with the following parameters: (a) $\{50, 60, 70, 80\}$ vertices; (b) edge probability for random and bipartite graphs and the euclidean distance for geometric graphs lie in $\{0.2, 0.4, 0.6, 0.8\}$.
 Five instances were generated for each combination of number of vertices and edge probability (or euclidean distance for the geometric graphs), therefore instances with those same characteristics, but different seeds, are organized into instance groups. Each instance group is identified by $C\_n\_p$, where $C$ represents the class of the graph: random ($R$), bipartite ($B$), and geometric ($G$); $n$ gives the number of vertices and $p$ denotes the edge probability for random and bipartite graphs, and the euclidean distance for geometric graphs. 
 \rafaelC{More challenging large bipartite and random instances were also created in a similar fashion but with the number of vertices in $\{500,600,700,800\}$.}
 We remark that all results reported for this set of instances represent average values over the corresponding instance group.

We also use the graphs presented in the benchmark instances from the Second DIMACS Implementation Challenge as they are largely used in the literature, especially for coloring and maximum clique problems \cite{AvHeZu03, LuHa10, MoGo18, NoPiSu18, SaCoFuLj19}. The instances are identified by their original filename and can be obtained in the DIMACS Implementation Challenges website~\cite{instances}. {We denote the instances for coloring problems as graph coloring instances and those for maximum clique problems as maximum clique instances.} {The complete benchmark instances along with detailed results for each instance 
are available in~\citeA{MelQueSan20} at Mendeley Data.}

\subsection{Tested approaches and parameters setting}\label{sec:testedapproaches}
In this subsection we present the tested approaches and the preliminary experiments carried out to determine the parameters of the proposed techniques. The following approaches were considered in the computational experiments:

\begin{itemize}
\item [(a)] MSBCOL: run exclusively MULTISTART-B-COL in parallel using all cores of the target machine;
\item [(b)] {MSBCOL$^+$: run the matheuristic}, using the best solution encountered by the metaheuristic MSBCOL as a warm start {for the MIP-based fix-and-optimize local search procedure};
 
\item [(c)] MSBCOL$^*$: run the complete integer programming formulation presented in Section~\ref{sec:formulation}, using the best solution encountered by the metaheuristic MSBCOL as a warm start. Following Observations~\ref{obs:fixzero}~and~\ref{obs:fixzero2}, variables corresponding to vertices with degree less than or equal to this best solution value are fixed to zero, as long as they are not $b$-vertices in the warm start solution;

\item [(d)] IP: Run the integer programming formulation presented in Section~\ref{sec:formulation} without any initial solution or fixings of variables.
\end{itemize}
The used test strategy was adopted to evaluate the behavior of the newly proposed methods according with the class and size of the benchmark instances. Furthermore, we wanted to verify the effectiveness of the {MIP-based fix-and-optimize local search} when compared with the complete formulation.

 Define $p(G)$ to be the density of $G$, calculated as $p(G) = \frac{2 \times |E|}{|V| \times (|V| - 1)}$, and let the maximum number of iterations for MULTISTART-B-COL, $it_{max}$, be computed as $it_{max} = 100 + \round{\frac{1000}{\sqrt{|V|} \times \sqrt{p(G)}}}$. {This formula for $it_{max}$ can be interpreted as follows.} {The minimum number of iterations that the algorithm executes is given by the first part of the formula, which is the constant 100. The variable number of iterations given by $\round{\frac{1000}{\sqrt{|V|} \times \sqrt{p(G)}}}$ is inversely proportional to the size and density of the graph, as iterations become more time consuming on larger and denser graphs. Such choice was made as an attempt to allow a reasonable number of iterations in order to avoid poor performance of the algorithm.}   
 The experiments to tune the parameter values are reported in the following. We randomly selected a small subset containing approximately 5.0\% of the instances with varying characteristics for parameter tuning. The following values were tested for each parameter:
\begin{itemize}
\item [(a)] $\alpha \in \{0.00, 0.05, 0.10, 0.15\}$; 
\item [(b)] $\beta \in \{0.00, 0.05, 0.10, 0.15\}$; 
\end{itemize}
The best obtained parameter values for MULTISTART-B-COL were $\alpha = 0.00$ and $\beta = 0.10$.

\subsection{Small instances}\label{sec:small}

Tables~\ref{table:bip1}-\ref{table:rand1} report the results for MSBCOL, MSBCOL$^+$, MSBCOL$^*$ and IP on the new set of generated \rafaelC{small} instances composed of bipartite, geometric and random graphs. The first column identifies the instance group. Columns 2 to 4 report the number of vertices ($|V|$), the average number of edges ($|E|$) along with the average solution upper bound for the instance group ($m(G)$).  Columns 5 to 8 give, for MSBCOL, the {best} encountered solution values ($z_M$), {the average solution values for the executed number of iterations ($z_{avg}$)}, the average running times in seconds ($time(s)$), and the percentual gap between the solution found by MSBCOL and the best obtained solution ($best$), calculated as $100 \times \frac{(best - z_M)}{best}$ ($\%_{best}$).  Columns 9 and 10 give, for MSBCOL$^+$, the encountered solution values ($z_{M^+}$) and the average running times in seconds ($time(s)$) for the MIP-based local search procedure. Columns 10 to 15 give, for the exact approaches MSBCOL$^*$ and IP, the encountered solution values ($z_{M^*}$ and $z_{IP}$, respectively), the average running times to solve the instances to optimality ($time(s)$), and the average open gaps (in \%) of the unsolved instances ($gap$), calculated as $100 \times \frac{(ub - lb)}{ub}$, where $lb$ represents the best known integer solution and $ub$ the best upper bound achieved at the end of the execution. The last two lines report the number of best known solutions found by each of the proposed approaches ($\# best$),  and, for MSBCOL$^*$ and IP, the amount of instances solved to optimality ($\# opt$). 


\rafaelC{The value '\textit{\color{blue}n/a}' in a cell indicates that, for at least one instance in the group, either the solver exceeded the time limit before obtaining a feasible solution or the execution was halted by the operating system due to memory limitations. 
The value '\textit{\color{black}t.l.}'  for column $time(s)$ means that none of the instances in the group were solved to optimality within the time limit  of 3,600 seconds using the corresponding integer program. 
The value '-'  for column $gap$ represents that all five instances in the group were solved to optimality.}
The best encountered solution values are shown in bold.

\begin{landscape}
\begin{table}[H]
\centering
\small
\caption{Results for MSBCOL conducted on small bipartite graphs.}
\begin{tabular}{lccc|cccc|cc|ccc|ccc}
\hline
\multicolumn{1}{l}{Instance group} & \multicolumn{3}{c|}{} & \multicolumn{4}{c|}{MSBCOL} & \multicolumn{2}{c|}{MSBCOL$^+$} & \multicolumn{3}{c|}{MSBCOL$^*$} & \multicolumn{3}{c}{IP}\\
 & $|V|$ & $|E|$ & $m(G)$ & $z_M$ & {$z_{avg}$} &  time(s) & $\%_{best}$ & $z_{M^+}$ & time(s) & $z_{M^*}$ & time(s) & gap(\%) & $z_{IP}$ & time(s) & gap(\%) \\
 \hline
bip\_50\_0.2	&	50	&	128.8	&	8.2	&	7.6	& {6.4} &	$<$0.1	&	7.3	&	7.8	&	$<$0.1	&	\textbf{8.2}	&	$<$0.1	&	-	&	\textbf{8.2}	&	17.6	&	-	\\			
bip\_50\_0.4	&	50	&	252.0	&	12.8	&	10.2 & {8.2} &	$<$0.1	&	15.0	&	11.8	&	0.6	&	\textbf{12.0}	&	5.2	&	-	&	\textbf{12.0}	&	50.6	&	-	\\			
bip\_50\_0.6	&	50	&	373.4	&	17.0	&	12.2	& {10.3} &	$<$0.1	&	17.6	&	14.0	&	1.0	&	\textbf{14.8}	&	840.6	&	-	&	\textbf{14.8}	&	583.8	&	6.7	\\			
bip\_50\_0.8	&	50	&	470.0	&	19.4	&	15.4	& {12.6} &	$<$0.1	&	9.4	&	16.6	&	$<$0.1	&	\textbf{17.0}	&	32.6	&	-	&	\textbf{17.0}	&	41.4	&	-	\\			
bip\_60\_0.2	&	60	&	181.2	&	9.6	&	8.8	& {7.5} &	$<$0.1	&	6.4	&	\textbf{9.4}	&	0.2	&	\textbf{9.4}	&	0.2	&	-	&	\textbf{9.4}	&	83.8	&	-	\\			
bip\_60\_0.4	&	60	&	351.6	&	14.8	&	11.2	& {9.4} &	$<$0.1	&	20.0	&	13.6	&	311.2	&	\textbf{14.0}	&	113.7	&	7.1	&	\textbf{14.0}	&	405.0	&	7.1	\\			
bip\_60\_0.6	&	60	&	526.2	&	19.8	&	14.4	& {11.9} &	$<$0.1	&	16.3	&	16.6	&	6.6	&	\textbf{17.2}	&	205.0	&	11.0	&	\textbf{17.2}	&	1,772.0	&	12.3	\\			
bip\_60\_0.8	&	60	&	709.0	&	24.8	&	16.8	& {11.5} &	$<$0.1	&	18.4	&	20.0	&	0.2	&	\textbf{20.6}	&	885.2	&	-	&	\textbf{20.6}	&	1,053.8	&	-	\\			
bip\_70\_0.2	&	70	&	246.2	&	10.4	&	9.4	& {8.0} &	$<$0.1	&	9.6	&	10.0	&	0.2	&	\textbf{10.4}	&	0.4	&	-	&	\textbf{10.4}	&	373.8	&	-	\\			
bip\_70\_0.4	&	70	&	468.6	&	17.4	&	12.6	& {10.7} &	0.1	&	19.2	&	\textit{\color{black}n/a} 	&	\textit{\color{black}n/a} 	&	\textit{\color{black}n/a} 	&	\textit{\color{black}n/a} 	&	\textit{\color{black}n/a} 	&	\textbf{15.2}	&	\textit{\color{black}t.l.}	&	13.2	\\
bip\_70\_0.6	&	70	&	700.4	&	23.0	&	15.8	& {13.1} &	0.1	&	17.7	&	\textbf{19.2}	&	2,337.2	&	18.6	&	\textit{\color{black}t.l.}	&	32.6	&	18.6	&	\textit{\color{black}t.l.}	&	26.4	\\			
bip\_70\_0.8	&	70	&	946.8	&	27.6	&	20.2	& {14.3} &	0.1	&	15.1	&	22.8	&	1.0	&	\textbf{23.8}&	60.5	&	14.5	&	\textbf{23.8}	&	853.0	&	14.1	\\				
bip\_80\_0.2	&	80	&	314.8	&	11.6	&	10.0	& {8.5} &	0.1	&	13.8	&	11.4	&	0.4	&	\textbf{11.6}	&	3.8	&	-	&	\textbf{11.6}	&	951.8	&	-	\\			
bip\_80\_0.4	&	80	&	644.8	&	19.2	&	13.4	& {11.2} &	0.1	&	19.3	&	\textbf{16.6}	&	901.8	&	16.4	&	\textit{\color{black}t.l.}	&	38.7	&	16.4	&	\textit{\color{black}t.l.}	&	107.8	\\			
bip\_80\_0.6	&	80	&	947.6	&	26.6	&	17.0	& {14.4} &	0.1	&	19.0	&	\textbf{21.0}	&	\textit{\color{black}t.l.}	&	20.0	&	\textit{\color{black}t.l.}	&	95.4	&	20.0	&	\textit{\color{black}t.l.}	&	97.0	\\			
bip\_80\_0.8	&	80	&	1,257.0	&	32.8	&	21.4	& {15.1} &	0.2	&	16.4	&	\textbf{25.6}	&	15.2	&	25.2	&	\textit{\color{black}t.l.}	&	30.5	&	24.8	&	\textit{\color{black}t.l.}	&	31.9\\				
\hline
\# best & & & & 0/16 & & & & 5/16 & & 11/16 & & & 12/16   &  \\ 
\# opt & & & & & & & & & & & & 47/80 & & & 45/80 \\ 
\hline
\end{tabular}
\label{table:bip1}
\end{table}

\begin{table}[H]
\centering
\small
\caption{Results for MSBCOL conducted on small geometric graphs.}
\begin{tabular}{lccc|cccc|cc|ccc|ccc}
\hline
\multicolumn{1}{l}{Instance group} & \multicolumn{3}{c|}{} & \multicolumn{4}{c|}{MSBCOL} & \multicolumn{2}{c|}{MSBCOL$^+$} & \multicolumn{3}{c|}{MSBCOL$^*$} & \multicolumn{3}{c}{IP}\\
 & $|V|$ & $|E|$ & $m(G)$ & $z_M$ & {$z_{avg}$} & time(s) & $\%_{best}$ & $z_{M^+}$ & time(s) & $z_{M^*}$ & time(s) & gap(\%) & $z_{IP}$ & time(s) & gap(\%) \\
 \hline
geo\_50\_0.2	&	50	&	125.8	&	8.8	&	8.4	& {7.9} &	$<$0.1	&	2.3	&	\textbf{8.6}	&	$<$0.1	&	\textbf{8.6}	&	$<$0.1	&	-	&	\textbf{8.6}	&	6.0	&	-	\\
geo\_50\_0.4	&	50	&	437.4	&	20.8	&	19.6	& {18.0} &	$<$0.1	&	4.9	&	20.4	&	$<$0.1	&	\textbf{20.6}	&	0.2	&	-	&	\textbf{20.6}	&	5.4	&	-	\\
geo\_50\_0.6	&	50	&	765.4	&	29.6	&	27.8	& {26.0} &	0.1	&	3.5	&	\textbf{28.8}	&	$<$0.1	&	\textbf{28.8}	&	0.2	&	-	&	\textbf{28.8}	&	1.2	&	-	\\
geo\_50\_0.8	&	50	&	1,002.4	&	36.4	&	34.2	& {33.2} &	$<$0.1	&	1.2	&	\textbf{34.6}	&	$<$0.1	&	\textbf{34.6}	&	$<$0.1	&	-	&	\textbf{34.6}	&	0.2	&	-	\\
geo\_60\_0.2	&	60	&	186.0	&	9.6	&	9.4	& {9.3} &	$<$0.1	&	2.1	&	9.4	&	$<$0.1	&	\textbf{9.6}	&	$<$0.1	&	-	&	\textbf{9.6}	&	19.4	&	-	\\
geo\_60\_0.4	&	60	&	637.8	&	24.2	&	21.6	& {19.8} &	0.1	&	10.7	&	24.0	&	$<$0.1	&	\textbf{24.2}	&	0.4	&	-	&	\textbf{24.2}	&	15.6	&	-	\\
geo\_60\_0.6	&	60	&	1,111.0	&	35.2	&	32.6	& {30.8} &	0.2	&	4.7	&	34.0	&	$<$0.1	&	\textbf{34.2}	&	0.2	&	-	&	\textbf{34.2}	&	2.4	&	-	\\
geo\_60\_0.8	&	60	&	1,494.6	&	45.8	&	42.6	& {41.2} &	0.1	&	1.4	&	\textbf{43.2}	&	$<$0.1	&	\textbf{43.2}	&	$<$0.1	&	-	&	\textbf{43.2}	&	$<$0.1	&	-	\\
geo\_70\_0.2	&	70	&	265.0	&	11.8	&	11.4	& {10.3} &	$<$0.1	&	3.4	&	11.6	&	$<$0.1	&	\textbf{11.8}	&	$<$0.1	&	-	&	\textbf{11.8}	&	66.4	&	-	\\
geo\_70\_0.4	&	70	&	812.0	&	26.4	&	23.4	& {21.3} &	0.1	&	11.4	&	26.0	&	$<$0.1	&	\textbf{26.4}	&	1.2	&	-	&	\textbf{26.4}	&	40.4	&	-	\\
geo\_70\_0.6	&	70	&	1,448.8	&	39.2	&	36.2	& {34.0} &	0.2	&	5.2	&	\textbf{38.2}	&	0.2	&	\textbf{38.2}	&	0.8	&	-	&	\textbf{38.2}	&	8.4	&	-	\\
geo\_70\_0.8	&	70	&	2,020.8	&	53.0	&	49.6	& {47.6} &	0.2	&	1.6	&	50.2	&	$<$0.1	&	\textbf{50.4}	&	0.2	&	-	&	\textbf{50.4}	&	0.4	&	-	\\
geo\_80\_0.2	&	80	&	331.2	&	12.4	&	11.8	& {10.8} &	$<$0.1	&	4.8	&	\textbf{12.4}	&	$<$0.1	&	\textbf{12.4}	&	$<$0.1	&	-	&	\textbf{12.4}	&	141.0	&	-	\\
geo\_80\_0.4	&	80	&	1,072.2	&	30.4	&	26.2	& {24.0} &	0.2	&	13.2	&	29.8	&	0.4	&	\textbf{30.2}	&	16.6	&	-	&	\textbf{30.2}	&	170.8	&	-	\\
geo\_80\_0.6	&	80	&	1,950.6	&	46.0	&	42.6	& {39.2} &	0.3	&	4.9	&	44.4	&	0.2	&	\textbf{44.8}	&	1.0	&	-	&	\textbf{44.8}	&	11.2	&	-	\\
geo\_80\_0.8	&	80	&	2,689.6	&	61.0	&	55.8	& {54.1} &	0.2	&	2.4	&	56.6	&	$<$0.1	&	\textbf{57.2}	&	0.2	&	-	&	\textbf{57.2}	&	0.8	&	-	\\
\hline
\# best & & & & 0/16 & & & & 6/16 & & 16/16 & & & 16/16  & &  \\
\# opt & & & & & & & & & & & & 80/80 & & & 80/80 \\
\hline
\end{tabular}
\label{table:geo1}
\end{table}

\begin{table}[H]
\centering
\small
\caption{Experiments conducted on small random graphs.}
\begin{tabular}{lccc|cccc|cc|ccc|ccc}
\hline
\multicolumn{1}{l}{Instance group} & \multicolumn{3}{c|}{} & \multicolumn{4}{c|}{MSBCOL} & \multicolumn{2}{c|}{MSBCOL$^+$} & \multicolumn{3}{c|}{MSBCOL$^*$} & \multicolumn{3}{c}{IP}\\
 & $|V|$ & $|E|$ & $m(G)$ & $z_M$ & {$z_{avg}$} &  time(s) & $\%_{best}$ & $z_{M^+}$ & time(s) & $z_{M^*}$ & time(s) & gap(\%) & $z_{IP}$ & time(s) & gap(\%) \\
 \hline
 
rand\_50\_0.2	&	50	&	248.8	&	12.8	&	9.6	& {8.1} &	$<$0.1	&	21.3	&	11.4	&	1.4	&	\textbf{12.2}	&	17.4	&	-	&	\textbf{12.2}	&	41.6	&	-	\\				
rand\_50\_0.4	&	50	&	485.0	&	21	&	13.0	& {11.3} &	$<$0.1	&	19.8	&	\textbf{15.8}	&	34.2	&	\textit{\color{black}n/a} 	&	\textit{\color{black}n/a} 	&	\textit{\color{black}n/a} 	&	\textit{\color{black}n/a} 	&	\textit{\color{black}n/a} 	&	\textit{\color{black}n/a} 	\\
rand\_50\_0.6	&	50	&	730.8	&	29.4	&	17.8	& {15.7} &	0.1	&	16.8	&	20.4	&	0.8	&	\textbf{21.4}	&	1299.5	&	8.5	&	21.2	&	837.0	&	8.8	\\				
rand\_50\_0.8	&	50	&	982.8	&	38	&	24.0	& {21.7} &	0.1	&	10.4	&	25.6	&	$<$0.1	&	\textbf{26.8}	&	8.2	&	-	&	\textbf{26.8}	&	6.0	&	-	\\				
rand\_60\_0.2	&	60	&	355.2	&	15	&	10.8	& {9.1} &	$<$0.1	&	21.7	&	13.4	&	12.2	&	\textbf{13.8}	&	799.25	&	7.1	&	\textbf{13.8}	&	667.5	&	7.3	\\				
rand\_60\_0.4	&	60	&	711.4	&	25.4	&	15.2	& {13.3} &	0.1	&	18.3	&	\textbf{18.6}	&	\textit{\color{black}t.l.}	&	18.2	&	\textit{\color{black}t.l.}	&	29.2	&	\textbf{18.6}	&	\textit{\color{black}t.l.}	&	26.9	\\				
rand\_60\_0.6	&	60	&	1065.0	&	35.8	&	20.6	& {18.1} &	0.1	&	15.6	&	\textbf{24.4}	&	12.0	&	24.2	&	\textit{\color{black}t.l.}	&	19.6	&	24.2	&	\textit{\color{black}t.l.}	&	19.9	\\				
rand\_60\_0.8	&	60	&	1426.4	&	46.2	&	27.8	& {25.3} &	0.1	&	12.6	&	30.6	&	$<$0.1	&	\textbf{31.8}	&	63.6	&	-	&	\textbf{31.8}	&	59.8	&	-	\\				
rand\_70\_0.2	&	70	&	477.6	&	16.8	&	11.8	& {9.9} &	0.1	&	21.3	&	\textbf{15.0}	&	2265.8	&	\textbf{15.0}	&	\textit{\color{black}t.l.}	&	8.7	&	14.8	&	\textit{\color{black}t.l.}	&	26.6	\\				
rand\_70\_0.4	&	70	&	984.4	&	30	&	17.0	& {15.0} &	0.2	&	19.8	&	\textbf{21.2}	&	\textit{\color{black}t.l.}	&	20.4	&	\textit{\color{black}t.l.}	&	42.0	&	20.4	&	\textit{\color{black}t.l.}	&	40.6	\\				
rand\_70\_0.6	&	70	&	1425.0	&	41.2	&	22.2	& {19.9} &	0.2	&	17.8	&	\textbf{27.0}	&	98.2	&	26.8	&	\textit{\color{black}t.l.}	&	30.8	&	\textbf{27.0}	&	\textit{\color{black}t.l.}	&	29.3	\\				
rand\_70\_0.8	&	70	&	1934.2	&	53.6	&	31.2	& {28.2} &	0.2	&	14.3	&	34.8	&	$<$0.1	&	\textbf{36.4}	&	659.4	&	-	&	\textbf{36.4}	&	846.8	&	-	\\				
rand\_80\_0.2	&	80	&	646.6	&	19.6	&	13.0	& {11.1} &	0.1	&	22.6	&	\textbf{16.8}	&	2916.2	&	16.2	&	\textit{\color{black}t.l.}	&	42.0	&	15.8	&	\textit{\color{black}t.l.}	&	81.0	\\				
rand\_80\_0.4	&	80	&	1266.6	&	33.8	&	18.2	& {16.4} &	0.2	&	18.0	&	\textbf{22.2}	&	\textit{\color{black}t.l.}	&	21.8	&	\textit{\color{black}t.l.}	&	68.9	&	21.4	&	\textit{\color{black}t.l.}	&	97.5	\\				
rand\_80\_0.6	&	80	&	1896.6	&	47.4	&	25.0	& {22.5} &	0.2	&	19.4	&	\textbf{31.0}	&	1127.0	&	29.6	&	\textit{\color{black}t.l.}	&	43.2	&	29.2	&	\textit{\color{black}t.l.}	&	43.3	\\				
rand\_80\_0.8	&	80	&	2521.8	&	61.8	&	35.0	& {31.2} &	0.2	&	14.2	&	39.4	&	$<$0.1	&	\textbf{40.8}	&	\textit{\color{black}t.l.}	&	6.3	&	\textbf{40.8}	&	\textit{\color{black}t.l.}	&	6.6	\\				
\hline
\# best & & & & 0/16 & & & & 9/16 & & 8/16 & & & 8/16  & &  \\
\# opt & & & & & & & & & & & & 26/80 & & & 23/80 \\
\hline
\end{tabular}
\label{table:rand1}
\end{table}

\end{landscape}

 The results show that MSBCOL performed very efficiently, as the reported running times for all instances in the three classes of graphs were below 0.4 seconds, which is practically negligible. Besides, the solutions encountered by MSBCOL are all within 23.0\% of the best encountered solution. Results are particularly noteworthy for geometric graphs, as 12 out of 16 (75.0\%) reported solutions were within 5.0\% of the best known. Regarding MSBCOL$^+$, the results show that the MIP-based fix-and-optimize local search has improved the initial solutions provided by MSBCOL for all 48 instance groups. MSBCOL$^+$ encountered the best known solutions for 6 out of 18 (33.3\%) instance groups for both geometric and bipartite graphs. Furthermore, MSBCOL$^+$ presented notable results for random graphs, as it reported 8 out of 16 (50.0\%) best known solutions.  MSBCOL$^+$ was also very effective, as 30 out of 48 instance groups (62.5\%) were solved in less than one second. The most impressive performance can be seen for geometric graphs, as the method was able to solve 13 out of 18 instance groups (72.2\%) in less than 0.1 seconds. 

Similarly to MSBCOL$^+$, MSBCOL$^*$ also improved the initial solutions provided by MSBCOL for all 48 instance groups. For bipartite graphs MSBCOL$^*$ found the best known solution for 12 out of 16 instance groups (75.0\%). Additionally for bipartite graphs, 47 out of 80 instances (58.8\%) were optimally solved by the approach.  MSBCOL$^*$ presented remarkable results for geometric graphs, reporting the best known and optimal solutions for all 16 instance groups. Regarding random instances, MSBCOL$^*$ performed well, encountering best known solution values for 8 out of 16 (50.0\%) instance groups. Furthermore, 26 out of 80 (31.5\%) small random instances were solved to optimality. Finally, within the given time limit, 28 out of 48 (58.3\%) instance groups were completely solved to optimally using MSBCOL$^*$. 

Lastly, IP presented similar results when compared with MSBCOL$^*$, as it reached optimality in 27 out of 48 (56.2\%) instance groups, however, MSBCOL$^*$ uses less computational times, which can be specially evidenced in groups $bip\_70\_0.2$, $bip\_80\_0.2$, $geo\_80\_0.2$,  and $geo\_80\_0.4$. The results show that for bipartite graphs IP obtained 11 out of 16 (68.7\%) best known solutions for bipartite graphs. As for geometric graphs, identical to MSBCOL$^*$, IP achieved the best known and optimal solutions for all 16 instance groups. IP also performed well on random instances, considering that it encountered best known solutions for 9 out of 16 (56.2\%) instance groups, and solved to optimality 23 out of 80 instances (28.8\%). 


Tables~\ref{table:dimacscol1}-\ref{table:dimacsclq1} report the results for MSBCOL, MSBCOL$^+$, MSBCOL$^*$ and IP on the set of small graphs from the Second DIMACS Implementation Challenge. The first column identify the DIMACS instance. Columns 2 to 4 report the number of vertices of each graph ($|V|$), the  number of edges ($|E|$) along with the solution upper bound ($m(G)$).  Columns 5 to 8 give, for MSBCOL and MSBCOL$^+$, the encountered solution values ($z_M$ and $z_{M^+}$, respectively) and the running time in seconds ($time(s)$). Columns 9 to 14 give, for MSBCOL$^*$ and IP, the encountered solution values ($z_{M^*}$ and $z_{IP}$, respectively), the running time to solve the instance ($time(s)$), and the open gap (in \%) in case of unsolved instances ($gap$), defined as before. The last two lines report the number of best known solutions found by each of the proposed approaches ($\# best$),  and, for MSBCOL$^*$ and IP, the amount of instances solved to optimality ($\# opt$).

\rafaelC{The value '\textit{\color{black}n/a}' in a cell expresses that either the solver exceeded the time limit before obtaining a feasible solution or the execution was halted by the operating system due to memory limitations. 
The value '\textit{\color{black}t.l.}'  for column $time(s)$ indicates that the instance was not solved to optimality within the time limit of 3,600 seconds. 
The value '-'  for column $gap$ means that the instance was solved to optimality.}
The best encountered solution values are shown in bold.

\begin{landscape}

\begin{table}
\centering
\small
\caption{Experiments conducted on small DIMACS graphs for coloring problems.}
\begin{tabular}{lccc|cccc|cc|ccc|ccc}
\hline
\multicolumn{1}{l}{Instance name} & \multicolumn{3}{c|}{} & \multicolumn{4}{c|}{MSBCOL} & \multicolumn{2}{c|}{MSBCOL$^+$} & \multicolumn{3}{c|}{MSBCOL$^*$} & \multicolumn{3}{c}{IP}\\
 & $|V|$ & $|E|$ & $m(G)$ & $z_M$ & {$z_{avg}$} & time(s) & $\%_{best}$ & $z_{M^+}$ & time(s) & $z_{M^*}$ & time(s) & gap(\%) & $z_{IP}$ & time(s) & gap(\%) \\
 \hline
dsjc125.1.col	&	125	&	736	&	17	&	13	& {11.0} &	0.2	&	23.5	&	15	&	11.0	&	\textbf{17}	&	794.0	&	-	&	8	&	\textit{\color{black}t.l.}	&	716.3	\\
dsjc125.5.col	&	125	&	3,891	&	63	&	30	& {27.4} &	0.6	&	14.3	&	\textbf{35}	&	\textit{\color{black}t.l.}	&	32	&	\textit{\color{black}t.l.}	&	126.8	&	34	&	\textit{\color{black}t.l.}	&	115.5	\\
dsjc125.9.col	&	125	&	6,961	&	109	&	62	& {57.7} &	0.6	&	8.8	&	65	&	$<$0.1	&	\textbf{68}	&	\textit{\color{black}t.l.}	&	7.0	&	\textbf{68}	&	\textit{\color{black}t.l.}	&	7.0	\\
dsjc250.1.col	&	250	&	3,218	&	33	&	19	& {17.0} & 	0.7	&	5.0	&	19	&	\textit{\color{black}t.l.}	&	19	&	\textit{\color{black}t.l.}	&	1105.3	&	\textbf{20}	&	\textit{\color{black}t.l.}	&	1150.0	\\
dsjr500.1.col	&	500	&	3,555	&	23	&	21	& {19.8} &	1.0	&	8.7	&	21	&	5.0	&	\textbf{23}	&	170.0	&	-	&	14	&	\textit{\color{black}t.l.}	&	3471.4	\\
fpsol2.i.2.col	&	451	&	8,691	&	53	&	52	& {51.6} &	9.9	&	1.9	&	52	&	3.0	&	\textbf{53}	&	5.0	&	-	&	30	&	\textit{\color{black}t.l.}	&	397.1	\\
fpsol2.i.3.col	&	425	&	8,688	&	53	&	52	& {51.6} &	9.4	&	1.9	&	52	&	2.0	&	\textbf{53}	&	4.0	&	-	&	\textit{\color{black}n/a}	&	\textit{\color{black}n/a}	&	\textit{\color{black}n/a}	\\
le450\_15a.col	&	450	&	8,168	&	57	&	\textbf{35}	& {29.8} &	3.7	&	0.0	&	\textbf{35}	&	\textit{\color{black}t.l.}	&	\textbf{35}	&	\textit{\color{black}t.l.}	&	270.6	&	21	&	\textit{\color{black}t.l.}	&	2042.9	\\
le450\_15b.col	&	450	&	8,169	&	56	&	\textbf{34}	& {29.7} &	3.5	&	0.0	&	\textbf{34}	&	\textit{\color{black}t.l.}	&	\textbf{34}	&	\textit{\color{black}t.l.}	&	309.4	&	20	&	\textit{\color{black}t.l.}	&	2150.0	\\
le450\_25a.col	&	450	&	8,260	&	63	&	\textbf{41}	& {36.5} &	4.8	&	0.0	&	\textbf{41}	&	\textit{\color{black}t.l.}	&	\textbf{41}	&	\textit{\color{black}t.l.}	&	166.7	&	\textit{\color{black}n/a}	&	\textit{\color{black}n/a}	&	\textit{\color{black}n/a}	\\
le450\_25b.col	&	450	&	8,263	&	60	&	\textbf{39}	& {35.2} &	4.1	&	0.0	&	\textbf{39}	&	\textit{\color{black}t.l.}	&	\textbf{39}	&	\textit{\color{black}t.l.}	&	200.6	&	\textit{\color{black}n/a}	&	\textit{\color{black}n/a}	&	\textit{\color{black}n/a}	\\
le450\_5a.col	&	450	&	5,714	&	34	&	\textbf{24}	& {20.1} &	1.6	&	0.0	&	\textbf{24}	&	\textit{\color{black}t.l.}	&	\textbf{24}	&	\textit{\color{black}t.l.}	&	1104.2	&	19	&	\textit{\color{black}t.l.}	&	2268.4	\\
le450\_5b.col	&	450	&	5,734	&	34	&	\textbf{22}	& {19.6} &	1.6	&	0.0	&	\textbf{22}	&	\textit{\color{black}t.l.}	&	\textbf{22}	&	\textit{\color{black}t.l.}	&	1472.7	&	14	&	\textit{\color{black}t.l.}	&	3114.3	\\
le450\_5c.col	&	450	&	9,803	&	52	&	27	& {24.0} &	2.4	&	15.6	&	27	&	\textit{\color{black}t.l.}	&	27	&	\textit{\color{black}t.l.}	&	1566.7	&	\textbf{32}	&	\textit{\color{black}t.l.}	&	1306.2	\\
le450\_5d.col	&	450	&	9,757	&	52	&	\textbf{26}	& {23.9} &	2.5	&	0.0	&	\textbf{26}	&	\textit{\color{black}t.l.}	&	\textbf{26}	&	\textit{\color{black}t.l.}	&	1630.8	&	21	&	\textit{\color{black}t.l.}	&	2042.9	\\
mulsol.i.1.col	&	197	&	3,925	&	65	&	60	& {58.3} &	0.8	&	6.3	&	60	&	1.0	&	\textbf{64}	&	11.0	&	-	&	\textbf{64}	&	1,546.0	&	-	\\
mulsol.i.2.col	&	188	&	3,885	&	53	&	39	& {39.0} &	1.1	&	23.5	&	50	&	1.0	&	\textbf{51}	&	1.0	&	-	&	\textbf{51}	&	2,737.0	&	-	\\
mulsol.i.3.col	&	184	&	3,916	&	54	&	51	& {44.8} &	1.1	&	1.9	&	51	&	$<$0.1	&	\textbf{52}	&	1.0	&	-	&	\textbf{52}	&	\textit{\color{black}t.l.}	&	10.3	\\
mulsol.i.4.col	&	185	&	3,946	&	54	&	\textbf{52}	& {51.5} &	1.1	&	0.0	&	\textbf{52}	&	$<$0.1	&	\textbf{52}	&	$<$0.1	&	-	&	\textbf{52}	&	\textit{\color{black}t.l.}	&	1.9	\\
mulsol.i.5.col	&	186	&	3,973	&	55	&	52	& {51.8} &	1.2	&	1.9	&	52	&	$<$0.1	&	\textbf{53}	&	1.0	&	-	&	\textbf{53}	&	1,702.0	&	-	\\
r125.1c.col	&	125	&	7,501	&	116	&	50	& {48.1} &	0.5	&	5.7	&	\textbf{53}	&	$<$0.1	&	\textbf{53}	&	$<$0.1	&	-	&	\textbf{53}	&	$<$0.1	&	-	\\
r125.1.col	&	125	&	209	&	7	&	\textbf{7}	& {6.1} &	$<$0.1	&	0.0	&	\textbf{7}	&	$<$0.1	&	\textbf{7}	&	$<$0.1	&	-	&	\textbf{7}	&	468.0	&	-	\\
r125.5.col	&	125	&	3,838	&	61	&	52	& {48.8} &	0.7	&	13.3	&	\textbf{60}	&	3.0	&	\textbf{60}	&	34.0	&	-	&	\textbf{60}	&	626.0	&	-	\\
r250.1.col	&	250	&	867	&	13	&	\textbf{12}	& {11.3} &	0.3	&	0.0	&	\textbf{12}	&	1.0	&	\textbf{12}	&	1.0	&	-	&	8	&	\textit{\color{black}t.l.}	&	1483.2	\\
zeroin.i.1.col	&	211	&	4,100	&	54	&	53	& {51.0} &	0.7	&	1.9	&	53	&	1.0	&	\textbf{54}	&	1.0	&	-	&	\textbf{54}	&	67.0	&	-	\\
zeroin.i.2.col	&	211	&	3,541	&	41	&	35	& {33.4} &	1.1	&	14.6	&	36	&	1.0	&	\textbf{41}	&	1.0	&	-	&	\textbf{41}	&	2,405.0	&	-	\\
zeroin.i.3.col	&	206	&	3,540	&	41	&	36	& {33.4} &	1.1	&	12.2	&	37	&	1.0	&	\textbf{41}	&	2.0	&	-	&	\textbf{41}	&	1,719.0	&	-	\\	\hline	
\# best & & & & 10/27 & & & & 13/27 & & 24/27 & & & 14/27  & &  \\
\# opt & & & & & & & & & & & & 16/27 & & & 9/27 \\
\hline
\end{tabular}
\label{table:dimacscol1}
\end{table}


\begin{table}
\centering
\small
\caption{Experiments conducted on small DIMACS graphs for the maximum clique problem.}
\begin{tabular}{lccc|cccc|cc|ccc|ccc}
\hline
\multicolumn{1}{l}{Instance name} & \multicolumn{3}{c|}{} & \multicolumn{4}{c|}{MSBCOL} & \multicolumn{2}{c|}{MSBCOL$^+$} & \multicolumn{3}{c|}{MSBCOL$^*$} & \multicolumn{3}{c}{IP} \\
 & $|V|$ & $|E|$ & $m(G)$ & $z_M$ & {$z_{avg}$} & time(s) & $\%_{best}$ & $z_{M^+}$ & time(s) & $z_{M^*}$ & time(s) & gap(\%) & $z_{IP}$ & time(s) & gap(\%) \\
 \hline
brock200\_2.clq	&	200	&	9,876	&	100	&	43	& {38.9} &	1.6	&	10.4	&	47	&	\textit{\color{black}t.l.}	&	43	&	\textit{\color{black}t.l.}	&	178.7	&	\textbf{48}	&	\textit{\color{black}t.l.}	&	149.7	\\
c125.9.clq	&	125	&	6,963	&	108	&	63	& {57.5} &	0.6	&	7.4	&	64	&	$<$0.1	&	\textbf{68}	&	\textit{\color{black}t.l.}	&	6.5	&	\textbf{68}	&	\textit{\color{black}t.l.}	&	6.8	\\
c-fat200-1.clq	&	200	&	1,534	&	18	&	\textbf{18}	& {18.0} &	$<$0.1	&	0.0	&	\textbf{18}	&	1.0	&	\textbf{18}	&	$<$0.1	&	-	&	\textbf{18}	&	\textit{\color{black}t.l.}	&	472.0	\\
c-fat200-2.clq	&	200	&	3,235	&	34	&	\textbf{34}	& {34.0} &	$<$0.1	&	0.0	&	\textbf{34}	&	1.0	&	\textbf{34}	&	$<$0.1	&	-	&	33	&	\textit{\color{black}t.l.}	&	225.2	\\
c-fat200-5.clq	&	200	&	8,473	&	86	&	\textbf{86}	& {86.0} &	$<$0.1	&	0.0	&	\textbf{86}	&	1.0	&	\textbf{86}	&	2.0	&	-	&	\textbf{86}	&	1,561.0	&	-	\\
c-fat500-1.clq	&	500	&	4,459	&	21	&	\textbf{21}	& {21.0} &	$<$0.1	&	0.0	&	\textbf{21}	&	4.0	&	\textbf{21}	&	4.0	&	-	&	18	&	\textit{\color{black}t.l.}	&	2677.8	\\
c-fat500-2.clq	&	500	&	9,139	&	39	&	\textbf{39}	& {39.0} &	$<$0.1	&	0.0	&	\textbf{39}	&	8.0	&	\textbf{39}	&	8.0	&	-	&	26	&	\textit{\color{black}t.l.}	&	1823.1	\\
hamming6-2.clq	&	64	&	1,824	&	58	&	34	& {32.4} &	0.1	&	2.9	&	34	&	$<$0.1	&	\textbf{35}	&	3.0	&	-	&	\textbf{35}	&	3.0	&	-	\\
hamming6-4.clq	&	64	&	704	&	23	&	13	& {10.4} &	0.1	&	13.3	&	\textbf{15}	&	\textit{\color{black}t.l.}	&	\textbf{15}	&	\textit{\color{black}t.l.}	&	53.3	&	\textbf{15}	&	\textit{\color{black}t.l.}	&	53.3	\\
johnson8-2-4.clq	&	28	&	210	&	16	&	\textbf{9}	& {6.4} &	$<$0.1	&	0.0	&	\textbf{9}	&	$<$0.1	&	\textbf{9}	&	2.0	&	-	&	\textbf{9}	&	4.0	&	-	\\
johnson8-4-4.clq	&	70	&	1,855	&	54	&	23	& {19.9} &	0.2	&	17.9	&	26	&	1.0	&	\textbf{28}	&	\textit{\color{black}t.l.}	&	31.0	&	26	&	\textit{\color{black}t.l.}	&	41.1	\\
johnson16-2-4.clq	&	120	&	5,460	&	92	&	17	& {14.6} &	0.6	&	54.1	&	17	&	$<$0.1	&	\textbf{37}	&	\textit{\color{black}t.l.}	&	32.8	&	36	&	\textit{\color{black}t.l.}	&	32.7	\\
keller4.clq	&	171	&	9,435	&	106	&	40	& {33.7} &	1.4	&	16.7	&	\textbf{48}	&	\textit{\color{black}t.l.}	&	43	&	\textit{\color{black}t.l.}	&	127.8	&	45	&	\textit{\color{black}t.l.}	&	117.0	\\
mann\_a9.clq	&	45	&	918	&	41	&	\textbf{21}	& {19.7} &	0.1	&	0.0	&	\textbf{21}	&	$<$0.1	&	\textbf{21}	&	$<$0.1	&	-	&	\textbf{21}	&	$<$0.1	&	-	\\
\hline
\# best & & & & 7/14 & & & & 9/14 & & 12/14 & & & 8/14  & &  \\
\# opt & & & & & & & & & & & & 8/14 & & & 4/14 \\
\hline
\end{tabular}
\label{table:dimacsclq1}
\end{table}

\end{landscape}

The multi-start approach, MSBCOL, achieved notable results, considering that for graph coloring instances (Table~\ref{table:dimacscol1}) the algorithm reported 10 out of 27 (37.0\%) best known results, and for maximum clique instances (Table~\ref{table:dimacsclq1}) returned 7 out of 14 (50.0\%) best known solution values. MSBCOL achieved the majority of solutions within 20.0\% of the best known, with a few exceptions being instances dsjc125.1.col, mulsol.i.2.col and johnson16-2-4.clq. Moreover, several of the reported solution values are within 5.0\% of the best known, as can be seen in dsjc250.1.col, fpsol2.i.2.col, fpsol2.i.3.col, mulsol.i.3.col, mulsol.i.5.col, zeroin.i.1.col, and hamming6-2.clq. Lastly, we also mention that these results were generated very efficiently, as the reported times were all under 10.0 seconds for graph coloring instances and 2.0 seconds for maximum clique instances. 

The results show that MSBCOL$^+$ improved the initial solutions provided by MSBCOL for 7 out of 27 (25.9\%) graph coloring instances and for 5 out of 14 (35.71\%) maximum clique instances. The most outstanding improvements can be seen in dsjc125.5.col, mulsol.i.2.col, brock2002.clq and keller4.clq. MSBCOL$^+$ encountered best known solutions for 13 out of 27 (48.2\%) graph coloring instances and 11 out of 14 (78.6\%) for maximum clique instances. 

One can see from the tables that MSBCOL$^*$ outperformed MSBCOL$^+$, especially for graph coloring instances, whereas the approach enhanced the initial solutions provided by MSBCOL in 15 out of 27 (55.6\%) cases. As for maximum clique instances, 6 out of 14 (42.9\%) initial solutions were improved. Note that MSBCOL$^*$ found the best known solution values for a majority of instances, which strongly supports its effectiveness. For graph coloring instances, 24 out of 27 (88.9\%) best known solutions were reported, as for maximum clique instances it returned 12 out of 14 (85.7\%) best known results. Additionally, MSBCOL$^*$ optimally solved 16 out of 27 (59.3\%) graph coloring instances, and 8 out of 14 (57.1\%) maximum clique instances.  

 The results also show that IP reported noticeable inferior results when compared to MSBCOL$^*$, as the approach returned best known solutions for 14 out of 27 (51.9\%) graph coloring instances and  8 out of 14 (57.1\%) maximum clique instances. Besides, IP did not obtain integer feasible solutions for three graph coloring instances (fpsol2.i.3.col, le450\_25a.col, le450\_25b.col), which reinforce the importance of the initial solutions provided by MSBCOL. Lastly, IP optimally solved 12 out of 27 (44.4\%) graph coloring instances, and 4 out of 14 (28.6\%) maximum clique instances. It is noteworthy that MSBCOL$^*$ solved instances to optimality considerably faster than IP as can be seen in instances mulsol.i.2.col and zeroin.i.2.col, which were solved by MSBCOL$^*$ in around 1.0 second, meanwhile IP took over 2000.0 seconds to solve them.

 Overall, the results show that MSBCOL can generate solutions very quickly, which is advantageous in cases where one values performance over optimality. Additionally, both MSBCOL$^+$ and MSBCOL$^*$ accomplished to improve MSBCOL results. Even though MSBCOL$^+$ was outperformed by MSBCOL$^*$ and IP in terms of solution values, it presents lower computational times and the idea could be heuristically adapted and used in a combinatorial local search strategy to achieve even better solutions. MSBCOL$^*$ and IP obtained better known solutions for the majority of instances, but it is worth mentioning that they are more viable options when larger computational times are available. One can observe that the optimal solutions found by MSBCOL$^*$ and IP show that for the tested small instances the $b$-chromatic number is equal or very close to the upper bound $m(G)$. Analyzing the performances of MSBCOL$^*$ and IP, MSBCOL$^*$ has much lower computational times in general and optimally solved more instances than IP, which suggests the usefulness of initial solutions provided by MSBCOL.


\subsection{Large instances}\label{sec:large}

\rafaelC{
Tables~\ref{table:largebip1} and \ref{table:largerand1} report the results for MSBCOL, MSBCOL$^+$, MSBCOL$^*$, and IP on the new set of more challenging instances composed of large bipartite and random graphs. The structure of these tables is similar to that of Tables~\ref{table:bip1}-\ref{table:rand1}. Note that large geometric instances were not tested as it was observed for the small instances in Section~\ref{sec:small} that they are much easier to solve than the bipartite and random ones. 
}

\begin{landscape}
\begin{table}[H]
\color{black}
\centering
\small
\caption{Results for MSBCOL conducted on large bipartite graphs.}
\begin{tabular}{lccc|cccc|cc|ccc|ccc}
\hline
\multicolumn{1}{l}{Instance group} & \multicolumn{3}{c|}{} & \multicolumn{4}{c|}{MSBCOL} & \multicolumn{2}{c|}{MSBCOL$^+$} & \multicolumn{3}{c|}{MSBCOL$^*$} & \multicolumn{3}{c}{IP}\\
 & $|V|$ & $|E|$ & $m(G)$ & $z_M$ & {$z_{avg}$} &  time(s) & $\%_{best}$ & $z_{M^+}$ & time(s) & $z_{M^*}$ & time(s) & gap(\%) & $z_{IP}$ & time(s) & gap(\%) \\
 \hline
bip\_500\_0.2 & 500 & 12,473.0  & 58.4  & 28.6  & 25.8 & 3.1  & 21.4 & {28.6} & \textit{\color{black}t.l.} & {28.6} & \textit{\color{black}t.l.} & {1,648.8} & \textbf{36.4} & \textit{\color{black}t.l.} & {1,274.7}  \\
bip\_500\_0.4 & 500 & 24,970.0  & 107.0 & \textbf{44.8}  & 39.9 & 5.3  & 0.0 & \textbf{44.8} & \textit{\color{black}t.l.} & \textbf{44.8} & \textit{\color{black}t.l.} & {1,016.2} & \textit{\color{black}n/a}                      & \textit{\color{black}n/a}                & \textit{\color{black}n/a}                          \\
bip\_500\_0.6 & 500 & 37,500.4  & 155.0 & \textbf{62.4}  & 49.3 & 7.5  & 0.0 & \textbf{62.4} & \textit{\color{black}t.l.} & \textbf{62.4} & \textit{\color{black}t.l.} & {701.4}  & \textit{\color{black}n/a}                      & \textit{\color{black}n/a}                & \textit{\color{black}n/a}                          \\
bip\_500\_0.8 & 500 & 49,961.0  & 202.8 & \textbf{72.4}  & 48.8 & 8.6  & 0.0 & \textit{\color{black}n/a}                      & \textit{\color{black}n/a}                & \textbf{72.4} & \textit{\color{black}t.l.} & {591.0}  & \textit{\color{black}n/a}                      & \textit{\color{black}n/a}                & \textit{\color{black}n/a}                          \\
bip\_600\_0.2 & 600 & 17,989.6  & 69.4  & \textbf{32.0}  & 29.5 & 4.5  & 0.0 & \textbf{32.0} & \textit{\color{black}t.l.} & \textbf{32.0} & \textit{\color{black}t.l.} & {1,775.0} & \textit{\color{black}n/a}                      & \textit{\color{black}n/a}                & \textit{\color{black}n/a}                          \\
bip\_600\_0.4 & 600 & 35,968.8  & 127.8 & \textbf{52.0}  & 45.7 & 8.4  & 0.0 & \textit{\color{black}n/a}                      & \textit{\color{black}n/a}                & \textit{\color{black}n/a}                      & \textit{\color{black}n/a}                & \textit{\color{black}n/a}                        & \textit{\color{black}n/a}                      & \textit{\color{black}n/a}                & \textit{\color{black}n/a}                          \\
bip\_600\_0.6 & 600 & 54,006.0  & 185.4 & \textbf{72.2}  & 57.3 & 11.4 & 0.0 & \textit{\color{black}n/a}                      & \textit{\color{black}n/a}                & \textit{\color{black}n/a}                      & \textit{\color{black}n/a}                & \textit{\color{black}n/a}                        & \textit{\color{black}n/a}                      & \textit{\color{black}n/a}                & \textit{\color{black}n/a}                          \\
bip\_600\_0.8 & 600 & 71,938.4  & 243.0 & \textbf{84.0}  & 56.5 & 13.4 & 0.0 & \textbf{84.0} & \textit{\color{black}t.l.} & \textbf{84.0} & \textit{\color{black}t.l.} & {614.9}  & \textit{\color{black}n/a}                      & \textit{\color{black}n/a}                & \textit{\color{black}n/a}                          \\
bip\_700\_0.2 & 700 & 24,559.2  & 80.6  & \textbf{36.2}  & 32.9 & 6.2  & 0.0 & \textit{\color{black}n/a}                      & \textit{\color{black}n/a}                & \textit{\color{black}n/a}                      & \textit{\color{black}n/a}                & \textit{\color{black}n/a}                        & \textit{\color{black}n/a}                      & \textit{\color{black}n/a}                & \textit{\color{black}n/a}                          \\
bip\_700\_0.4 & 700 & 49,056.2  & 149.2 & \textbf{58.6}  & 51.7 & 11.2 & 0.0 & \textit{\color{black}n/a}                      & \textit{\color{black}n/a}                & \textit{\color{black}n/a}                      & \textit{\color{black}n/a}                & \textit{\color{black}n/a}                        & \textit{\color{black}n/a}                      & \textit{\color{black}n/a}                & \textit{\color{black}n/a}                          \\
bip\_700\_0.6 & 700 & 73,554.2  & 217.0 & \textbf{82.2}  & 63.5 & 16.9 & 0.0 & \textit{\color{black}n/a}                      & \textit{\color{black}n/a}                & \textit{\color{black}n/a}                      & \textit{\color{black}n/a}                & \textit{\color{black}n/a}                        & \textit{\color{black}n/a}                      & \textit{\color{black}n/a}                & \textit{\color{black}n/a}                          \\
bip\_700\_0.8 & 700 & 97,935.4  & 283.8 & \textbf{98.0}  & 63.9 & 20.0 & 0.0 & \textit{\color{black}n/a}                      & \textit{\color{black}n/a}                & \textit{\color{black}n/a}                      & \textit{\color{black}n/a}                & \textit{\color{black}n/a}                        & \textit{\color{black}n/a}                      & \textit{\color{black}n/a}                & \textit{\color{black}n/a}                          \\
bip\_800\_0.2 & 800 & 32,040.8  & 90.8  & \textbf{39.2}  & 35.9 & 8.3  & 0.0 & \textit{\color{black}n/a}                      & \textit{\color{black}n/a}                & \textit{\color{black}n/a}                      & \textit{\color{black}n/a}                & \textit{\color{black}n/a}                        & \textit{\color{black}n/a}                      & \textit{\color{black}n/a}                & \textit{\color{black}n/a}                          \\
bip\_800\_0.4 & 800 & 63,974.2  & 169.2 & \textbf{64.4}  & 56.0 & 16.7 & 0.0 & \textit{\color{black}n/a}                      & \textit{\color{black}n/a}                & \textit{\color{black}n/a}                      & \textit{\color{black}n/a}                & \textit{\color{black}n/a}                        & \textit{\color{black}n/a}                      & \textit{\color{black}n/a}                & \textit{\color{black}n/a}                          \\
bip\_800\_0.6 & 800 & 95,985.8  & 246.6 & \textbf{91.0}  & 71.4 & 22.6 & 0.0 & \textit{\color{black}n/a}                      & \textit{\color{black}n/a}                & \textit{\color{black}n/a}                      & \textit{\color{black}n/a}                & \textit{\color{black}n/a}                        & \textit{\color{black}n/a}                      & \textit{\color{black}n/a}                & \textit{\color{black}n/a}                          \\
bip\_800\_0.8 & 800 & 127,880.0 & 324.0 & \textbf{109.4} & 69.5 & 27.4 & 0.0 & \textit{\color{black}n/a}                      & \textit{\color{black}n/a}                & \textit{\color{black}n/a}                      & \textit{\color{black}n/a}                & \textit{\color{black}n/a}                        & \textit{\color{black}n/a}                      & \textit{\color{black}n/a}                & \textit{\color{black}n/a}                        \\
\hline
\# best & & & & 15/16 & & & & 4/16 & & 5/16 & & & 1/16   &  \\ 
\# opt & & & & & & & & & & & & 0/16 & & & 0/16 \\ 
\hline
\end{tabular}
\label{table:largebip1}
\end{table}

\begin{table}[H]
\color{black}
\centering
\small
\caption{Experiments conducted on large random graphs.}
\begin{tabular}{lccc|cccc|cc|ccc|ccc}
\hline
\multicolumn{1}{l}{Instance group} & \multicolumn{3}{c|}{} & \multicolumn{4}{c|}{MSBCOL} & \multicolumn{2}{c|}{MSBCOL$^+$} & \multicolumn{3}{c|}{MSBCOL$^*$} & \multicolumn{3}{c}{IP}\\
 & $|V|$ & $|E|$ & $m(G)$ & $z_M$ & {$z_{avg}$} &  time(s) & $\%_{best}$ & $z_{M^+}$ & time(s) & $z_{M^*}$ & time(s) & gap(\%) & $z_{IP}$ & time(s) & gap(\%) \\
 \hline
rand\_500\_0.2 & 500 & 24,964.4  & 107.6 & 42.8  & 39.4  & 5.8  & 15.7 & 42.8                    & \textit{\color{black}t.l.} & 42.8                    & \textit{\color{black}t.l.}         & 1,068.3                  & \textbf{50.8}  & \textit{\color{black}t.l.} & {897.9} \\
rand\_500\_0.4 & 500 & 49,845.8  & 203.0 & 71.8  & 67.1  & 11.8 & 12.0 & 71.8                    & \textit{\color{black}t.l.} & 71.8                    & \textit{\color{black}t.l.}         & 596.4                   & \textbf{81.6}  & \textit{\color{black}t.l.} & {513.0} \\
rand\_500\_0.6 & 500 & 74,751.0  & 297.6 & \textbf{105.6} & 99.4  & 17.2 & 0.0 & \textbf{105.6}                   & \textit{\color{black}t.l.} & {\textit{\color{black}n/a}} & \textit{\color{black}n/a}                        & {\textit{\color{black}n/a}} & \textit{\color{black}n/a}                       & \textit{\color{black}n/a}                & \textit{\color{black}n/a}                       \\
rand\_500\_0.8 & 500 & 99,755.2  & 393.2 & 155.4 & 145.8 & 19.0 & 0.6 & \textbf{156.4}                   & \textit{\color{black}t.l.} & 155.4                   & \textit{\color{black}t.l.}         & 106.8                   & {146.0} & \textit{\color{black}t.l.} & {123.5} \\
rand\_600\_0.2 & 600 & 36,019.0  & 128.6 & \textbf{49.0}  & 45.5  & 8.9  & 0.0 & \textbf{49.0}                     & \textit{\color{black}t.l.} & \textbf{49.0}                    & \textit{\color{black}t.l.}         & 1,124.5                  & \textit{\color{black}n/a}                       & \textit{\color{black}n/a}                & \textit{\color{black}n/a}                       \\
rand\_600\_0.4 & 600 & 71,805.4  & 243.2 & \textbf{83.0}  & 77.9  & 18.3 & 0.0 & \textbf{83.0}                      & \textit{\color{black}t.l.} & \textbf{83.0}                    & \textit{\color{black}t.l.}         & 622.9                   & \textit{\color{black}n/a}                       & \textit{\color{black}n/a}                & \textit{\color{black}n/a}                       \\
rand\_600\_0.6 & 600 & 107,681.0 & 357.2 & 122.8 & 116.1 & 26.0 & 3.9 & \textit{\color{black}n/a} & \textit{\color{black}n/a}                & 122.8                   & \textit{\color{black}t.l.}         & 388.6                   & \textbf{127.8} & \textit{\color{black}t.l.} & {371.8} \\
rand\_600\_0.8 & 600 & 143,733.0 & 472.4 & \textbf{181.2} & 170.8 & 29.5 & 0.0 & \textbf{181.2}                   & \textit{\color{black}t.l.} & \textbf{181.2}                  & \textit{\color{black}t.l.}         & 231.1                   & \textit{\color{black}n/a}                       & \textit{\color{black}n/a}                & \textit{\color{black}n/a}                       \\
rand\_700\_0.2 & 700 & 48,992.8  & 149.0 & \textbf{55.0}  & 50.9  & 12.4 & 0.0 & {\textit{\color{black}n/a}} & \textit{\color{black}n/a}                & {\textit{\color{black}n/a}} & \textit{\color{black}n/a}                        & {\textit{\color{black}n/a}} & \textit{\color{black}n/a}                       & \textit{\color{black}n/a}                & \textit{\color{black}n/a}                       \\
rand\_700\_0.4 & 700 & 97,795.2  & 283.4 & \textbf{94.0}  & 88.5  & 25.8 & 0.0 & {\textit{\color{black}n/a}} & \textit{\color{black}n/a}                & {\textit{\color{black}n/a}} & \textit{\color{black}n/a}                        & {\textit{\color{black}n/a}} & \textit{\color{black}n/a}                       & \textit{\color{black}n/a}                & \textit{\color{black}n/a}                       \\
rand\_700\_0.6 & 700 & 146,642.0 & 417.2 & \textbf{139.4} & 132.0 & 39.2 & 0.0 & \textbf{139.4}                   & \textit{\color{black}t.l.} & \textbf{139.4}                   & \textit{\color{black}t.l.}         & 402.2                   & \textit{\color{black}n/a}                       & \textit{\color{black}n/a}                & \textit{\color{black}n/a}                       \\
rand\_700\_0.8 & 700 & 195,652.0 & 551.6 & \textbf{206.4} & 194.6 & 44.1 & 0.0 & \textbf{206.4}                   & \textit{\color{black}t.l.} & \textbf{206.4}                   & \textit{\color{black}t.l.}         & 239.1                   & \textit{\color{black}n/a}                       & \textit{\color{black}n/a}                & \textit{\color{black}n/a}                       \\
rand\_800\_0.2 & 800 & 63,928.6  & 169.4 & \textbf{60.6}  & 56.3  & 16.7 & 0.0 & {\textit{\color{black}n/a}} & \textit{\color{black}n/a}                &{\textit{\color{black}n/a}} & \textit{\color{black}n/a}                        & {\textit{\color{black}n/a}} & \textit{\color{black}n/a}                       & \textit{\color{black}n/a}                & \textit{\color{black}n/a}                       \\
rand\_800\_0.4 & 800 & 127,726.0 & 323.6 & \textbf{105.0} & 98.0  & 38.6 & 0.0 & {\textit{\color{black}n/a}} & \textit{\color{black}n/a}                & {\textit{\color{black}n/a}} & \textit{\color{black}n/a}                        & {\textit{\color{black}n/a}} & \textit{\color{black}n/a}                       & \textit{\color{black}n/a}                & \textit{\color{black}n/a}                       \\
rand\_800\_0.6 & 800 & 191,557.0 & 476.8 & \textbf{156.0} & 147.9 & 54.1 & 0.0 & \textbf{156.0}                   & \textit{\color{black}t.l.} & {\textit{\color{black}n/a}} & \textit{\color{black}n/a}                        & {\textit{\color{black}n/a}} & \textit{\color{black}n/a}                       & \textit{\color{black}n/a}                & \textit{\color{black}n/a}                       \\
rand\_800\_0.8 & 800 & 255,578.0 & 630.6 & \textbf{230.2} & 218.0 & 63.9 & 0.0 & \textbf{230.2} & \textit{\color{black}t.l.} & \textbf{230.2} & \textit{\color{black}t.l.} & 247.5                   & \textit{\color{black}n/a}                       & \textit{\color{black}n/a}                & \textit{\color{black}n/a}                      \\
\hline
\# best & & & & 12/16 & & & & 9/16 & & 6/16 & & & 3/16  & &  \\
\# opt & & & & & & & & & & & & 0/16 & & & 0/16 \\
\hline
\end{tabular}
\label{table:largerand1}
\end{table}

\end{landscape}

\rafaelC{
The results show that even though the instances are considerably large, MSBCOL can still generate solutions in low computation times. More specifically, its running time was within 30 seconds for the bipartite and within 64 seconds for the random graphs. 
On the other hand, the results also show that, differently from what happened for the small bipartite and random instances, MSBCOL$^+$ and MSBCOL$^*$ were not able to consistently improve the quality of the solutions obtained by MSBCOL. Besides, for all the executions of MSBCOL$^+$, MSBCOL$^*$, and IP, either the time limit was reached or the execution was halted due to memory limitations.
In that sense, the empirical results show that these two new sets of instances appear to be very challenging for the $b$-coloring problem.
}

\rafaelC{Furthermore, note that IP was only able to finish its execution without being halted for bipartite instances with 500 vertices and random instances  with less than 600 vertices. It is noteworthy that, for most of the cases in which IP reached the time limit, the obtained solutions could improve those achieved by MSBCOL, the only exception being random\_500\_0.8. This observation leads to the conclusion that, even though MSBCOL can generate solutions of reasonable quality for these large challenging instances in low computational times, there is still space for improvements.}

Tables~\ref{table:dimacscol2}-\ref{table:dimacsclq4} report the results for MSBCOL, MSBCOL$^+$, MSBCOL$^*$ and IP on the set of large graphs from the Second DIMACS Implementation Challenge. The structure of these tables is identical to Tables~\ref{table:dimacscol1}-\ref{table:dimacsclq1}.

The results show that MSBCOL presented very good results when compared to the other approaches, as the multi-start approach achieved 15 out of 32 (46.9\%) best known solutions for graph coloring instances (Table~\ref{table:dimacscol2}). Regarding maximum clique instances (Table~\ref{table:dimacsclq4}), MSBCOL obtained 34 out of 64 (53.1\%) best known values. We highlight instances in which the reported gaps of the solutions were within 5.0\% of the best known: dsjc500.5.col, flat300\_28\_0.col, inithx.i.1.col, inithx.i.2.col, inithx.i.3.col, le450\_15d.col, r1000.1.col, brock400\_1.clq, brock400\_2.clq, brock400\_3.clq, brock400\_4.clq, hamming10-2.clq, hamming10-4.clq, hammin8-2.clq, san200\_0.9\_2.clq, san200\_0.9\_3.clq, and san400\_0.9\_1.clq. For the remaining instances, the solutions found by MSBCOL were all within 28.0\% of the best reported values. The reported times for graph coloring instances were all below 3.0 minutes, which is an impressive performance considering that the largest instance in this set (R1000.1c.col) has 1000 vertices and 485,090 edges. Additionally, MSBCOL also executed in less than 3.0 minutes for most maximum clique instances,  with the few exceptions being graphs whose number of vertices are at least 1500 or number of edges are over 800,000 (C2000.5.clq, C2000.9.clq, C4000.5.clq, keller6.clq, MANNa81.clq, phat1500-1.clq, phat1500-2.clq and phat1500-3.clq).

The results show that MSBCOL$^+$ reasonably improved solutions from MSBCOL, as the method enhanced 10 out of 32 (31.3\%) solutions for graph coloring instances, and with respect to maximum clique instances, MSBCOL$^+$ improved 26 out of 64 (37.5\%) solution values. The most remarkable improvements obtained by MSBCOL$^+$ can be seen in DSJC500.9.col, DSJR500.5.col, R1000.1c.col, C500.9.clq, gen400\_p0.9\_55.clq, gen400\_p0.9\_65.clq and gen400\_p0.9\_75.clq. For the previous mentioned instances, MSBCOL$^+$ returned a solution with at least 30 colors more when compared with the initial provided by MSBCOL, which is a strong indication of the advantage in applying such method. Moreover, MSBCOL$^+$ encountered best known solutions for 16 out of 32 (50.0\%) graph coloring instances and 37 out of 64 (57.8\%) for maximum clique instances.
 
Contrasting with the previous behaviour for small instances, when applied to large instances MSBCOL$^+$ presented slightly superior results than MSBCOL$^*$, as the latter was successful in improving solutions for 8 out of 32 (25.0\%) graph coloring instances and 18 out of 64 (28.1\%) maximum clique instances. In terms of solution values, MSBCOL$^*$ returned best known results for 15 out of 32 (46.9\%) graph coloring instances and 33 out of 64 (51.6\%) maximum clique instances. Results also show that MSBCOL$^*$ displayed difficulty in solving larger instances to optimality, as the method only solved 5 for each graph coloring (15.6\%) and maximum clique (7.8\%) instances. These results indicate the difficulty of the MIP solver in solving a problem when the number of variables increase substantially, which explains better results when the fix-and-optimize approach MSBCOL$^+$ was employed.

\begin{landscape}

\small 
\begin{longtable}{lccc|cccc|cc|ccc|ccc}
\caption{Experiments conducted on large DIMACS graphs for coloring problems.}
\label{table:dimacscol2}
\\
\hline
\multicolumn{1}{l}{Instance name} & \multicolumn{3}{c|}{} & \multicolumn{4}{c|}{MSBCOL} & \multicolumn{2}{c|}{MSBCOL$^+$} & \multicolumn{3}{c|}{MSBCOL$^*$} & \multicolumn{3}{c}{IP} \\
 & $|V|$ & $|E|$ & $m(G)$ & $z_M$ & {$z_{avg}$} & time(s) & $\%_{best}$ & $z_{M^+}$ & time(s) & $z_{M^*}$ & time(s) & gap(\%) & $z_{IP}$ & time(s) & gap(\%) \\
 \hline
dsjc1000.1.col	&	1,000	&	49,629	&	112	&	\textbf{44}	& {41.6} &	12.2	&	0.0	&	\textit{\color{black}n/a}	&	\textit{\color{black}n/a}	&	\textit{\color{black}n/a}	&	\textit{\color{black}n/a}	&	\textit{\color{black}n/a}	&	\textit{\color{black}n/a}	&	\textit{\color{black}n/a}	&	\textit{\color{black}n/a}	\\
dsjc1000.5.col	&	1,000	&	249,826	&	501	&	\textbf{154}	& {147} &	74.7	&	0.0	&	\textit{\color{black}n/a}	&	\textit{\color{black}n/a}	&	\textit{\color{black}n/a}	&	\textit{\color{black}n/a}	&	\textit{\color{black}n/a}	&	\textit{\color{black}n/a}	&	\textit{\color{black}n/a}	&	\textit{\color{black}n/a}	\\
dsjc1000.9.col	&	1,000	&	449,449	&	888	&	\textbf{349}	& {336} &	96.2	&	0.0	&	\textbf{349}	&	\textit{\color{black}t.l.}	&	\textbf{349}	&	\textit{\color{black}t.l.}	&	186.5	&	\textit{\color{black}n/a}	&	\textit{\color{black}n/a}	&	\textit{\color{black}n/a}	\\
dsjc250.5.col	&	250	&	15,668	&	126	&	51	& {46.9} &	2.8	&	10.5	&	51	&	\textit{\color{black}t.l.}	&	51	&	\textit{\color{black}t.l.}	&	194.3	&	\textbf{57}	&	\textit{\color{black}t.l.}	&	163.4	\\
dsjc250.9.col	&	250	&	27,897	&	219	&	110	& {101.8} &	3.0	&	13.4	&	125	&	2.0	&	\textbf{127}	&	\textit{\color{black}t.l.}	&	22.8	&	\textbf{127}	&	\textit{\color{black}t.l.}	&	21.9	\\
dsjc500.1.col	&	500	&	12,458	&	59	&	29	& {25.6} &	2.9	&	19.4	&	29	&	\textit{\color{black}t.l.}	&	29	&	\textit{\color{black}t.l.}	&	1624.1	&	\textbf{36}	&	\textit{\color{black}t.l.}	&	1288.9	\\
dsjc500.5.col	&	500	&	62,624	&	251	&	88	& {82.8} &	14.1	&	7.4	&	88	&	\textit{\color{black}t.l.}	&	88	&	\textit{\color{black}t.l.}	&	468.2	&	\textbf{95}	&	\textit{\color{black}t.l.}	&	426.3	\\
dsjc500.9.col	&	500	&	112,437	&	443	&	196	& {184.4} &	16.4	&	21.3	&	\textbf{249}	&	427.0	&	196	&	\textit{\color{black}t.l.}	&	67.1	&	180	&	\textit{\color{black}t.l.}	&	82.0	\\
dsjr500.1c.col	&	500	&	121,275	&	478	&	114	& {108.6} &	19.3	&	19.7	&	135	&	1.0	&	\textbf{142}	&	\textit{\color{black}t.l.}	&	43.4	&	141	&	\textit{\color{black}t.l.}	&	45.6	\\
dsjr500.5.col	&	500	&	58,862	&	234	&	183	& {19.8} &	24.3	&	17.2	&	\textbf{221}	&	\textit{\color{black}t.l.}	&	183	&	\textit{\color{black}t.l.}	&	112.6	&	150	&	\textit{\color{black}t.l.}	&	233.3	\\
flat1000\_50\_0.col	&	1,000	&	245,000	&	492	&	\textbf{153}	& {144.1} &	69.7	&	0.0	&	\textit{\color{black}n/a}	&	\textit{\color{black}n/a}	&	\textit{\color{black}n/a}	&	\textit{\color{black}n/a}	&	\textit{\color{black}n/a}	&	\textit{\color{black}n/a}	&	\textit{\color{black}n/a}	&	\textit{\color{black}n/a}	\\
flat1000\_60\_0.col	&	1,000	&	245,830	&	493	&	\textbf{152}	& {144.6} &	71.1	&	0.0	&	\textit{\color{black}n/a}	&	\textit{\color{black}n/a}	&	\textit{\color{black}n/a}	&	\textit{\color{black}n/a}	&	\textit{\color{black}n/a}	&	\textit{\color{black}n/a}	&	\textit{\color{black}n/a}	&	\textit{\color{black}n/a}	\\
flat1000\_76\_0.col	&	1,000	&	246,708	&	494	&	\textbf{153}	& {145.7} &	71.9	&	0.0	&	\textit{\color{black}n/a}	&	\textit{\color{black}n/a}	&	\textit{\color{black}n/a}	&	\textit{\color{black}n/a}	&	\textit{\color{black}n/a}	&	\textit{\color{black}n/a}	&	\textit{\color{black}n/a}	&	\textit{\color{black}n/a}	\\
flat300\_20\_0.col	&	300	&	21,375	&	144	&	\textbf{56}	& {51.9} &	3.8	&	0.0	&	\textbf{56}	&	\textit{\color{black}t.l.}	&	\textbf{56}	&	\textit{\color{black}t.l.}	&	219.5	&	\textbf{56}	&	\textit{\color{black}t.l.}	&	219.5	\\
flat300\_26\_0.col	&	300	&	21,633	&	146	&	57	& {52.6} &	3.7	&	10.9	&	57	&	\textit{\color{black}t.l.}	&	57	&	\textit{\color{black}t.l.}	&	214.4	&	\textbf{64}	&	\textit{\color{black}t.l.}	&	180.0	\\
flat300\_28\_0.col	&	300	&	21,695	&	146	&	57	& {52.8} &	3.9	&	9.5	&	57	&	\textit{\color{black}t.l.}	&	57	&	\textit{\color{black}t.l.}	&	214.5	&	\textbf{63}	&	\textit{\color{black}t.l.}	&	376.2	\\
fpsol2.i.1.col	&	496	&	11,654	&	79	&	\textbf{77}	& {75.0} &	5.8	&	0.0	&	\textbf{77}	&	7.0	&	\textbf{77}	&	7.0	&	-	&	65	&	\textit{\color{black}t.l.}	&	84.9	\\
inithx.i.1.col	&	864	&	18,707	&	74	&	71	& {70.8} &	32.6	&	1.4	&	71	&	30.0	&	\textbf{72}	&	26.0	&	-	&	\textit{\color{black}n/a}	&	\textit{\color{black}n/a}	&	\textit{\color{black}n/a}	\\
inithx.i.2.col	&	645	&	13,979	&	52	&	49	& {49.0} &	29.1	&	2.0	&	\textbf{50}	&	9.0	&	\textbf{50}	&	10.0	&	-	&	\textit{\color{black}n/a}	&	\textit{\color{black}n/a}	&	\textit{\color{black}n/a}	\\
inithx.i.3.col	&	621	&	13,969	&	52	&	49	& {49.0} & 	27.9	&	2.0	&	\textbf{50}	&	9.0	&	\textbf{50}	&	10.0	&	-	&	38	&	\textit{\color{black}t.l.}	&	1371.0	\\
latin\_square\_10.col	&	900	&	307,350	&	684	&	\textbf{183}	& {174.3} &	75.4	&	0.0	&	\textbf{183}	&	\textit{\color{black}t.l.}	&	\textbf{183}	&	\textit{\color{black}t.l.}	&	391.8	&	\textit{\color{black}n/a}	&	\textit{\color{black}n/a}	&	\textit{\color{black}n/a}	\\
le450\_15c.col	&	450	&	16,680	&	93	&	\textbf{41}	& {38.4} &	6.4	&	0.0	&	\textbf{41}	&	\textit{\color{black}t.l.}	&	\textbf{41}	&	\textit{\color{black}t.l.}	&	948.8	&	34	&	\textit{\color{black}t.l.}	&	1223.5	\\
le450\_15d.col	&	450	&	16,750	&	92	&	42	& {38.9} &	6.2	&	8.7	&	42	&	\textit{\color{black}t.l.}	&	42	&	\textit{\color{black}t.l.}	&	911.9	&	\textbf{46}	&	\textit{\color{black}t.l.}	&	878.3	\\
le450\_25c.col	&	450	&	17,343	&	101	&	\textbf{48}	& {44.6} &	8.7	&	0.0	&	\textbf{48}	&	\textit{\color{black}t.l.}	&	\textbf{48}	&	\textit{\color{black}t.l.}	&	687.5	&	40	&	\textit{\color{black}t.l.}	&	1025.0	\\
le450\_25d.col	&	450	&	17,425	&	99	&	\textbf{48}	& {44.9} &	7.2	&	0.0	&	\textbf{48}	&	\textit{\color{black}t.l.}	&	\textbf{48}	&	\textit{\color{black}t.l.}	&	693.8	&	\textit{\color{black}n/a}	&	\textit{\color{black}n/a}	&	\textit{\color{black}n/a}	\\
r1000.1c.col	&	1,000	&	485,090	&	957	&	156	& {146.0} &	154.6	&	27.4	&	\textbf{215}	&	\textit{\color{black}t.l.}	&	163	&	\textit{\color{black}t.l.}	&	183.0	&	162	&	\textit{\color{black}t.l.}	&	181.8	\\
r1000.1.col	&	1,000	&	14,378	&	41	&	35	& {32.1} &	4.0	&	2.8	&	\textbf{36}	&	316.0	&	35	&	\textit{\color{black}t.l.}	&	217.1	&	\textit{\color{black}n/a}	&	\textit{\color{black}n/a}	&	\textit{\color{black}n/a}	\\
r1000.5.col	&	1,000	&	238,267	&	472	&	\textbf{368}	& {360.2} &	168.1	&	0.0	&	\textit{\color{black}n/a}	&	\textit{\color{black}n/a}	&	\textit{\color{black}n/a}	&	\textit{\color{black}n/a}	&	\textit{\color{black}n/a}	&	\textit{\color{black}n/a}	&	\textit{\color{black}n/a}	&	\textit{\color{black}n/a}	\\
r250.1c.col	&	250	&	30,227	&	238	&	75	& {71.3} &	3.1	&	12.8	&	81	&	\textit{\color{black}t.l.}	&	\textbf{86}	&	58.0	&	-	&	\textbf{86}	&	35.0	&	-	\\
r250.5.col	&	250	&	14,849	&	119	&	100	& {94.9} &	3.7	&	13.8	&	\textbf{116}	&	3328.0	&	112	&	\textit{\color{black}t.l.}	&	11.6	&	102	&	\textit{\color{black}t.l.}	&	46.5	\\
school1.col	&	385	&	19,095	&	117	&	\textbf{58}	& {52.5} &	14.1	&	0.0	&	\textbf{58}	&	\textit{\color{black}t.l.}	&	\textbf{58}	&	\textit{\color{black}t.l.}	&	443.1	&	\textit{\color{black}n/a}	&	\textit{\color{black}n/a}	&	\textit{\color{black}n/a}	\\
school1\_nsh.col	&	352	&	14,612	&	101	&	\textbf{49}	& {46.2} &	9.2	&	0.0	&	\textbf{49}	&	\textit{\color{black}t.l.}	&	\textbf{49}	&	\textit{\color{black}t.l.}	&	504.1	&	36	&	\textit{\color{black}t.l.}	&	877.8	\\ \hline
\# best & & & & 15/32 & & & & 16/32 & & 15/32 & & & 9/32 & &  \\
\# opt & & & & & & & & & & & & 5/32 & & & 1/32 \\ 
 \hline
\end{longtable}
\end{landscape}

\begin{landscape}
\scriptsize
\begin{longtable}{lccc|cccc|cc|ccc|ccc}
\caption{Experiments conducted on large DIMACS graphs for the maximum clique problem.}
\label{table:dimacsclq4}
\\
\hline
\multicolumn{1}{l}{Instance name} & \multicolumn{3}{c|}{} & \multicolumn{4}{c|}{MSBCOL} & \multicolumn{2}{c|}{MSBCOL$^+$} & \multicolumn{3}{c|}{MSBCOL$^*$} & \multicolumn{3}{c}{IP} \\ 
 & $|V|$ & $|E|$ & $m(G)$ & $z_M$ & {$z_{avg}$} &  time(s) & $\%_{best}$ & $z_{M^+}$ & time(s) & $z_{M^*}$ & time(s) & gap(\%) & $z_{IP}$ & time(s) & gap(\%) \\ \hline \endfirsthead
 
 \multicolumn{15}{c}%
{{\tablename\ \thetable{}: continued from previous page}} \\
 \hline
\multicolumn{1}{l}{Instance name} & \multicolumn{3}{c|}{} & \multicolumn{4}{c|}{MSBCOL} & \multicolumn{2}{c|}{MSBCOL$^+$} & \multicolumn{3}{c|}{MSBCOL$^*$} & \multicolumn{3}{c}{IP} \\ 
 & $|V|$ & $|E|$ & $m(G)$ & $z_M$ & {$z_{avg}$} & time(s) & $\%_{best}$ & $z_{M^+}$ & time(s) & $z_{M^*}$ & time(s) & gap(\%) & $z_{IP}$ & time(s) & gap(\%) \\ \hline \endhead
 
 \hline \multicolumn{15}{r}{{Continued on next page}} \\ \hline
\endfoot

\endlastfoot

 brock200\_1.clq	&	200	&	14,834	&	146	&	64	& {60.2} &	2.1	&	12.3	&	\textbf{73}	&	\textit{\color{black}t.l.}	&	67	&	\textit{\color{black}t.l.}	&	83.8	&	70	&	\textit{\color{black}t.l.}	&	76.7	\\
brock200\_3.clq	&	200	&	12,048	&	120	&	51	& {47.3} &	1.8	&	10.5	&	\textbf{57}	&	\textit{\color{black}t.l.}	&	51	&	\textit{\color{black}t.l.}	&	139.4	&	56	&	\textit{\color{black}t.l.}	&	115.1	\\
brock200\_4.clq	&	200	&	13,089	&	129	&	56	& {51.7} &	2.0	&	11.1	&	62	&	\textit{\color{black}t.l.}	&	60	&	\textit{\color{black}t.l.}	&	103.2	&	\textbf{63}	&	\textit{\color{black}t.l.}	&	93.3	\\
brock400\_1.clq	&	400	&	59,723	&	294	&	115	& {109.1} &	10.0	&	6.5	&	\textbf{123}	&	\textit{\color{black}t.l.}	&	115	&	\textit{\color{black}t.l.}	&	121.3	&	119	&	\textit{\color{black}t.l.}	&	113.8	\\
brock400\_2.clq	&	400	&	59,786	&	295	&	115	& {108.3} &	11.0	&	5.0	&	115	&	\textit{\color{black}t.l.}	&	115	&	\textit{\color{black}t.l.}	&	121.3	&	\textbf{121}	&	\textit{\color{black}t.l.}	&	110.3	\\
brock400\_3.clq	&	400	&	59,681	&	294	&	116	& {108.4} &	10.6	&	5.7	&	116	&	\textit{\color{black}t.l.}	&	116	&	\textit{\color{black}t.l.}	&	119.3	&	\textbf{123}	&	\textit{\color{black}t.l.}	&	106.8	\\
brock400\_4.clq	&	400	&	59,765	&	295	&	116	& {108.7} &	10.8	&	7.2	&	116	&	\textit{\color{black}t.l.}	&	116	&	\textit{\color{black}t.l.}	&	119.4	&	\textbf{125}	&	\textit{\color{black}t.l.}	&	103.6	\\
brock800\_1.clq	&	800	&	207,505	&	515	&	\textbf{172}	& {161.8} &	57.9	&	0.0	&	\textbf{172}	&	\textit{\color{black}t.l.}	&	\textbf{172}	&	\textit{\color{black}t.l.}	&	365.1	&	\textit{\color{black}n/a}	&	\textit{\color{black}n/a}	&	\textit{\color{black}n/a}	\\
brock800\_2.clq	&	800	&	208,166	&	517	&	\textbf{172}	& {162.8} &	57.4	&	0.0	&	\textbf{172}	&	\textit{\color{black}t.l.}	&	\textbf{172}	&	\textit{\color{black}t.l.}	&	365.1	&	\textit{\color{black}n/a}	&	\textit{\color{black}n/a}	&	\textit{\color{black}n/a}	\\
brock800\_3.clq	&	800	&	207,333	&	515	&	\textbf{171}	& {161.7} &	56.7	&	0.0	&	\textbf{171}	&	\textit{\color{black}t.l.}	&	\textbf{171}	&	\textit{\color{black}t.l.}	&	367.8	&	\textit{\color{black}n/a}	&	\textit{\color{black}n/a}	&	\textit{\color{black}n/a}	\\
brock800\_4.clq	&	800	&	207,643	&	515	&	\textbf{171}	& {162.8} &	56.2	&	0.0	&	\textbf{171}	&	\textit{\color{black}t.l.}	&	\textbf{171}	&	\textit{\color{black}t.l.}	&	367.8	&	\textit{\color{black}n/a}	&	\textit{\color{black}n/a}	&	\textit{\color{black}n/a}	\\
C1000.9.clq	&	1,000	&	450,079	&	889	&	\textbf{348}	& {336.7} &	96.9	&	0.0	&	\textbf{348}	&	\textit{\color{black}t.l.}	&	\textbf{348}	&	\textit{\color{black}t.l.}	&	187.4	&	\textit{\color{black}n/a}	&	\textit{\color{black}n/a}	&	\textit{\color{black}n/a}	\\
C2000.5.clq	&	2,000	&	999,836	&	1,000	&	\textbf{277}	& {263.8} &	487.8	&	0.0	&	\textit{\color{black}n/a}	&	\textit{\color{black}n/a}	&	\textit{\color{black}n/a}	&	\textit{\color{black}n/a}	&	\textit{\color{black}n/a}	&	\textit{\color{black}n/a}	&	\textit{\color{black}n/a}	&	\textit{\color{black}n/a}	\\
C2000.9.clq	&	2,000	&	1,799,532	&	1,784	&	\textbf{650}	& {630.3} &	701.8	&	0.0	&	\textit{\color{black}n/a}	&	\textit{\color{black}n/a}	&	\textit{\color{black}n/a}	&	\textit{\color{black}n/a}	&	\textit{\color{black}n/a}	&	\textit{\color{black}n/a}	&	\textit{\color{black}n/a}	&	\textit{\color{black}n/a}	\\
C250.9.clq	&	250	&	27,984	&	220	&	111	& {102.3} &	3.0	&	12.6	&	124	&	2.0	&	\textbf{127}	&	\textit{\color{black}t.l.}	&	22.4	&	\textbf{127}	&	\textit{\color{black}t.l.}	&	22.8	\\
C4000.5.clq	&	4,000	&	4,000,268	&	2,002	&	\textbf{496}	& {477.3} &	3494.9	&	0.0	&	\textit{\color{black}n/a}	&	\textit{\color{black}n/a}	&	\textit{\color{black}n/a}	&	\textit{\color{black}n/a}	&	\textit{\color{black}n/a}	&	\textit{\color{black}n/a}	&	\textit{\color{black}n/a}	&	\textit{\color{black}n/a}	\\
C500.9.clq	&	500	&	112,332	&	442	&	195	& {182.9} &	17.5	&	22.0	&	\textbf{250}	&	\textit{\color{black}t.l.}	&	195	&	\textit{\color{black}t.l.}	&	68.0	&	180	&	\textit{\color{black}t.l.}	&	82.0	\\
c-fat500-10.clq	&	500	&	46,627	&	188	&	\textbf{188}	& {188.0} &	$<$0.1	&	0.0	&	\textbf{188}	&	27.0	&	\textbf{188}	&	29.0	&	-	&	126	&	\textit{\color{black}t.l.}	&	296.8	\\
c-fat500-5.clq	&	500	&	23,191	&	95	&	\textbf{95}	& {95.0} &	$<$0.1	&	0.0	&	\textbf{95}	&	18.0	&	\textbf{95}	&	19.0	&	-	&	65	&	\textit{\color{black}t.l.}	&	669.2	\\
gen200\_p0.9\_44.clq	&	200	&	17,910	&	174	&	87	& {79.7} &	1.9	&	16.3	&	100	&	$<$0.1	&	103	&	\textit{\color{black}t.l.}	&	19.4	&	\textbf{104}	&	\textit{\color{black}t.l.}	&	19.1	\\
gen200\_p0.9\_55.clq	&	200	&	17,910	&	174	&	90	& {82.6} &	1.8	&	13.5	&	100	&	$<$0.1	&	\textbf{104}	&	\textit{\color{black}t.l.}	&	18.6	&	\textbf{104}	&	\textit{\color{black}t.l.}	&	18.5	\\
gen400\_p0.9\_55.clq	&	400	&	71,820	&	348	&	161	& {142.4} &	10.0	&	19.5	&	\textbf{200}	&	23.0	&	193	&	\textit{\color{black}t.l.}	&	29.5	&	193	&	\textit{\color{black}t.l.}	&	28.9	\\
gen400\_p0.9\_65.clq	&	400	&	71,820	&	350	&	168	& {148.8} &	9.9	&	16.0	&	\textbf{200}	&	9.0	&	199	&	\textit{\color{black}t.l.}	&	25.9	&	\textbf{200}	&	\textit{\color{black}t.l.}	&	24.6	\\
gen400\_p0.9\_75.clq	&	400	&	71,820	&	350	&	163	& {151.6} &	9.8	&	18.5	&	\textbf{200}	&	16.0	&	198	&	\textit{\color{black}t.l.}	&	26.0	&	\textbf{200}	&	\textit{\color{black}t.l.}	&	24.8	\\
hamming10-2.clq	&	1,024	&	518,656	&	1,014	&	522	& {517.0} &	73.0	&	3.5	&	524	&	1.0	&	\textbf{541}	&	\textit{\color{black}t.l.}	&	23.6	&	531	&	\textit{\color{black}t.l.}	&	25.9	\\
hamming10-4.clq	&	1,024	&	434,176	&	849	&	144	& {132.8} &	148.1	&	8.3	&	144	&	\textit{\color{black}t.l.}	&	144	&	\textit{\color{black}t.l.}	&	611.1	&	\textbf{157}	&	\textit{\color{black}t.l.}	&	552.2	\\
hamming8-2.clq	&	256	&	31,616	&	248	&	132	& {129.5} &	2.2	&	8.3	&	134	&	$<$0.1	&	\textbf{144}	&	\textit{\color{black}t.l.}	&	11.1	&	\textbf{144}	&	\textit{\color{black}t.l.}	&	10.6	\\
hamming8-4.clq	&	256	&	20,864	&	164	&	42	& {37.5} &	3.0	&	12.5	&	47	&	\textit{\color{black}t.l.}	&	42	&	\textit{\color{black}t.l.}	&	257.3	&	\textbf{48}	&	\textit{\color{black}t.l.}	&	212.6	\\
johnson32-2-4.clq	&	496	&	107,880	&	436	&	34	& {30.6} &	20.8	&	15.0	&	\textbf{40}	&	3.0	&	34	&	\textit{\color{black}t.l.}	&	771.3	&	30	&	\textit{\color{black}t.l.}	&	887.5	\\
keller5.clq	&	776	&	225,990	&	565	&	\textbf{124}	& {104.7} &	73.9	&	0.0	&	\textbf{124}	&	\textit{\color{black}t.l.}	&	\textbf{124}	&	\textit{\color{black}t.l.}	&	525.8	&	89	&	\textit{\color{black}t.l.}	&	771.9	\\
keller6.clq	&	3,361	&	4,619,898	&	2,696	&	\textbf{332}	& {261.4} &	5554.9	&	0.0	&	\textit{\color{black}n/a}	&	\textit{\color{black}n/a}	&	\textit{\color{black}n/a}	&	\textit{\color{black}n/a}	&	\textit{\color{black}n/a}	&	\textit{\color{black}n/a}	&	\textit{\color{black}n/a}	&	\textit{\color{black}n/a}	\\ 
MANN\_a27.clq	&	378	&	70,551	&	365	&	\textbf{144}	& {141.8} &	6.4	&	0.0	&	\textbf{144}	&	$<$0.1	&	\textbf{144}	&	3.0	&	-	&	\textbf{144}	&	7.0	&	-	\\
MANN\_a45.clq	&	1,035	&	533,115	&	1,013	&	\textbf{375}	& {372.2} &	89.2	&	0.0	&	\textbf{375}	&	1.0	&	\textbf{375}	&	16.0	&	-	&	\textbf{375}	&	16.0	&	-	\\
MANN\_a81.clq	&	3,321	&	5,506,380	&	3,281	&	\textbf{1,161}	& {1157.2} &	2617.8	&	0.0	&	\textbf{1,161}	&	$<$0.1	&	\textbf{1,161}	&	75.0	&	-	&	\textbf{1,161}	&	80.0	&	-	\\
p\_hat1000-1.clq	&	1,000	&	122,253	&	298	&	\textbf{98}	& {91.4} &	68.6	&	0.0	&	\textit{\color{black}n/a}	&	\textit{\color{black}n/a}	&	\textit{\color{black}n/a}	&	\textit{\color{black}n/a}	&	\textit{\color{black}n/a}	&	\textit{\color{black}n/a}	&	\textit{\color{black}n/a}	&	\textit{\color{black}n/a}	\\
p\_hat1000-2.clq	&	1,000	&	244,799	&	496	&	\textbf{196}	& {187.2} &	175.2	&	0.0	&	\textit{\color{black}n/a}	&	\textit{\color{black}n/a}	&	\textit{\color{black}n/a}	&	\textit{\color{black}n/a}	&	\textit{\color{black}n/a}	&	\textit{\color{black}n/a}	&	\textit{\color{black}n/a}	&	\textit{\color{black}n/a}	\\
p\_hat1000-3.clq	&	1,000	&	371,746	&	694	&	\textbf{271}	& {257.0} &	149.1	&	0.0	&	\textbf{271}	&	\textit{\color{black}t.l.}	&	\textbf{271}	&	\textit{\color{black}t.l.}	&	269.0	&	\textit{\color{black}n/a}	&	\textit{\color{black}n/a}	&	\textit{\color{black}n/a}	\\
p\_hat1500-1.clq	&	1,500	&	284,923	&	457	&	\textbf{138}	& {130.5} &	202.9	&	0.0	&	\textit{\color{black}n/a}	&	\textit{\color{black}n/a}	&	\textit{\color{black}n/a}	&	\textit{\color{black}n/a}	&	\textit{\color{black}n/a}	&	\textit{\color{black}n/a}	&	\textit{\color{black}n/a}	&	\textit{\color{black}n/a}	\\
p\_hat1500-2.clq	&	1,500	&	568,960	&	760	&	\textbf{290}	& {277.9} &	545.2	&	0.0	&	\textit{\color{black}n/a}	&	\textit{\color{black}n/a}	&	\textit{\color{black}n/a}	&	\textit{\color{black}n/a}	&	\textit{\color{black}n/a}	&	\textit{\color{black}n/a}	&	\textit{\color{black}n/a}	&	\textit{\color{black}n/a}	\\
p\_hat1500-3.clq	&	1,500	&	847,244	&	1,051	&	\textbf{391}	& {375.1} &	458.7	&	0.0	&	\textit{\color{black}n/a}	&	\textit{\color{black}n/a}	&	\textit{\color{black}n/a}	&	\textit{\color{black}n/a}	&	\textit{\color{black}n/a}	&	\textit{\color{black}n/a}	&	\textit{\color{black}n/a}	&	\textit{\color{black}n/a}	\\
p\_hat300-1.clq	&	300	&	10,933	&	91	&	\textbf{39}	&	 {35.1} & 3.6	&	0.0	&	\textbf{39}	&	\textit{\color{black}t.l.}	&	\textbf{39}	&	\textit{\color{black}t.l.}	&	623.1	&	32	&	\textit{\color{black}t.l.}	&	837.5	\\
p\_hat300-2.clq	&	300	&	21,928	&	149	& 	\textbf{71}	& {66.4} &	7.3	&	0.0	&	\textbf{71}	&	\textit{\color{black}t.l.}	&	\textbf{71}	&	\textit{\color{black}t.l.}	&	148.2	&	69	&	\textit{\color{black}t.l.}	&	157.1	\\
p\_hat300-3.clq	&	300	&	33,390	&	209	&	95	& {90.1} &	6.3	&	15.9	&	107	&	\textit{\color{black}t.l.}	&	95	&	\textit{\color{black}t.l.}	&	100.3	&	\textbf{113}	&	\textit{\color{black}t.l.}	&	67.4	\\
p\_hat500-1.clq	&	500	&	31,569	&	152	&	\textbf{57}	&	 {53.2} & 11.6	&	0.0	&	\textbf{57}	&	\textit{\color{black}t.l.}	&	\textbf{57}	&	\textit{\color{black}t.l.}	&	773.7	&	49	&	\textit{\color{black}t.l.}	&	920.4	\\
p\_hat500-2.clq	&	500	&	62,946	&	252	&	\textbf{114}	& {105.8} &	27.9	&	0.0	&	\textbf{114}	&	\textit{\color{black}t.l.}	&	\textbf{114}	&	\textit{\color{black}t.l.}	&	338.6	&	95	&	\textit{\color{black}t.l.}	&	426.3	\\
p\_hat500-3.clq	& 	500	&	93,800	&	351	&	\textbf{151}	& {143.3} &	24.5	&	0.0	&	\textbf{151}	&	\textit{\color{black}t.l.}	&	\textbf{151}	&	\textit{\color{black}t.l.}	&	231.1	&	140	&	\textit{\color{black}t.l.}	&	257.1	\\
p\_hat700-1.clq	&	700	&	60,999	&	208	&	\textbf{74}	& {69.2} &	27.4	&	0.0	&	\textbf{74}	&	\textit{\color{black}t.l.}	&	\textbf{74}	&	\textit{\color{black}t.l.}	&	846.0	&	\textit{\color{black}n/a}	&	\textit{\color{black}n/a}	&	\textit{\color{black}n/a}	\\
p\_hat700-2.clq	&	700	&	121,728	&	353	&	\textbf{153}	& {141.6} &	66.5	&	0.0	&	\textbf{153}	&	\textit{\color{black}t.l.}	&	\textbf{153}	&	\textit{\color{black}t.l.}	&	357.5	&	\textit{\color{black}n/a}	&	\textit{\color{black}n/a}	&	\textit{\color{black}n/a}	\\
p\_hat700-3.clq	&	700	&	183,010	&	487	&	\textbf{201}	& {190.9} &	57.6	&	0.0	&	\textbf{201}	&	\textit{\color{black}t.l.}	&	\textbf{201}	&	\textit{\color{black}t.l.}	&	248.3	&	175	&	\textit{\color{black}t.l.}	&	300.0	\\
san1000.clq	&	1,000	&	250,500	&	514	&	\textbf{69}	& {59.9} &	103.6	&	0.0	&	\textit{\color{black}n/a}	&	\textit{\color{black}n/a}	&	\textit{\color{black}n/a}	&	\textit{\color{black}n/a}	&	\textit{\color{black}n/a}	&	\textit{\color{black}n/a}	&	\textit{\color{black}n/a}	&	\textit{\color{black}n/a}	\\
san200\_0.7\_1.clq	&	200	&	13,930	&	138	&	61	& {53.7} &	2.1	&	25.6	&	\textbf{82}	&	\textit{\color{black}t.l.}	&	63	&	\textit{\color{black}t.l.}	&	89.2	&	65	&	\textit{\color{black}t.l.}	&	83.3	\\
san200\_0.7\_2.clq	&	200	&	13,930	&	134	&	44	& {37.9} &	2.6	&	26.7	&	\textbf{60}	&	\textit{\color{black}t.l.}	&	48	&	\textit{\color{black}t.l.}	&	109.0	&	48	&	\textit{\color{black}t.l.}	&	109.7	\\
san200\_0.9\_1.clq	&	200	&	17,910	&	173	&	93	& {87.8} &	1.9	&	11.4	&	96	&	$<$0.1	&	\textbf{105}	&	\textit{\color{black}t.l.}	&	4.5	&	\textbf{105}	&	\textit{\color{black}t.l.}	&	4.5	\\
san200\_0.9\_2.clq	&	200	&	17,910	&	175	&	98	& {89.2} &	1.9	&	6.7	&	100	&	$<$0.1	&	\textbf{105}	&	\textit{\color{black}t.l.}	&	15.8	&	\textbf{105}	&	\textit{\color{black}t.l.}	&	15.9	\\
san200\_0.9\_3.clq	&	200	&	17,910	&	176	&	93	& {82.6} &	1.8	&	9.7	&	100	&	1.0	&	\textbf{103}	&	\textit{\color{black}t.l.}	&	19.4	&	\textbf{103}	&	\textit{\color{black}t.l.}	&	19.4	\\
san400\_0.5\_1.clq	&	400	&	39,900	&	204	&	\textbf{41}	& {34.9} &	10.0	&	0.0	&	\textbf{41}	&	\textit{\color{black}t.l.}	&	\textbf{41}	&	\textit{\color{black}t.l.}	&	875.6	&	37	&	\textit{\color{black}t.l.}	&	981.1	\\
san400\_0.7\_1.clq	&	400	&	55,860	&	277	&	94	& {86.3} &	11.3	&	16.8	&	\textbf{113}	&	\textit{\color{black}t.l.}	&	94	&	\textit{\color{black}t.l.}	&	163.0	&	101	&	\textit{\color{black}t.l.}	&	145.2	\\
san400\_0.7\_2.clq	&	400	&	55,860	&	277	&	87	& {77.9} &	13.0	&	19.4	&	\textbf{108}	&	\textit{\color{black}t.l.}	&	87	&	\textit{\color{black}t.l.}	&	185.4	&	78	&	\textit{\color{black}t.l.}	&	220.4	\\
san400\_0.7\_3.clq	&	400	&	55,860	&	274	&	\textbf{80}	& {70.8} &	14.8	&	0.0	&	\textbf{80}	&	\textit{\color{black}t.l.}	&	\textbf{80}	&	\textit{\color{black}t.l.}	&	209.5	&	78	&	\textit{\color{black}t.l.}	&	217.5	\\
san400\_0.9\_1.clq	&	400	&	71,820	&	353	&	190	& {172.3} &	9.1	&	6.4	&	200	&	1.0	&	\textbf{203}	&	\textit{\color{black}t.l.}	&	24.0	&	\textbf{203}	&	\textit{\color{black}t.l.}	&	24.7	\\
sanr200\_0.7.clq	&	200	&	13,868	&	137	&	60	& {55.6} &	2.2	&	10.4	&	\textbf{67}	&	\textit{\color{black}t.l.}	&	63	&	\textit{\color{black}t.l.}	&	95.2	&	64	&	\textit{\color{black}t.l.}	&	91.9	\\
sanr200\_0.9.clq	&	200	&	17,863	&	175	&	92	& {85.4} &	1.8	&	10.7	&	100	&	$<$0.1	&	\textbf{103}	&	\textit{\color{black}t.l.}	&	20.4	&	\textbf{103}	&	\textit{\color{black}t.l.}	&	20.4	\\
sanr400\_0.5.clq	&	400	&	39,984	&	201	&	\textbf{74}	& {68.8} &	8.4	&	0.0	&	\textbf{74}	&	\textit{\color{black}t.l.}	&	\textbf{74}	&	\textit{\color{black}t.l.}	&	440.5	&	\textit{\color{black}n/a}	&	\textit{\color{black}n/a}	&	\textit{\color{black}n/a}	\\
sanr400\_0.7.clq	&	400	&	55,869	&	276	&	\textbf{105}	& {98.9} &	10.3	&	0.0	&	\textbf{105}	&	\textit{\color{black}t.l.}	&	\textbf{105}	&	\textit{\color{black}t.l.}	&	139.8	&	\textit{\color{black}n/a}	&	\textit{\color{black}n/a}	&	\textit{\color{black}n/a}	\\ \hline
\# best & & & & {34/64} & & & & {37/64} & & {33/64} & & & {20/64}  & &  \\
\# opt & & & & & & & & & & & & {5/64} & & & {3/64} \\ 
 \hline
\end{longtable}
\end{landscape}

The results for IP on larger instances show once more that initial solutions provided by MSBCOL are very relevant, considering that both approaches that use those solutions as warm start, i.e. MSBCOL$^+$ and MSBCOL$^*$, reported higher numbers of best known solution values. IP returned 9 out of 32 (28.1\%) best known results for graph coloring instances, and 21 out of 64 (32.8\%) for maximum clique instances. Besides, the number of optimal solutions reported by IP is 1 for graph coloring instances (3.1\%), and 3 for maximum clique instances (4.7\%).

Generally speaking, the results in this section have reinforced the effectiveness of MSBCOL, as the approach obtained several best known values and reported reasonable running times even for very large graphs. Besides, MSBCOL was the only method that reported solutions for numerous instances, as can be seen in dsjc1000.1.col, as an example. Nevertheless, solutions generated by MSBCOL still have room for improvements, as results by MSBCOL$^+$ and MSBCOL$^*$ have shown. We also point out that a few optimal solutions reported by MSBCOL$^*$ and IP largely differ from the upper bound $m(G)$ (see MANN\_a27.clq, MANN\_a45.clq, and MANN\_a81.clq), however an in depth analysis is out of the scope of this work and can be an interesting direction for future theoretical works regarding lower bounds on the $b$-chromatic number. For the large set of instances, the heuristic methods (MSBCOL and MSBCOL$^+$) outperformed the exact ones (MSBCOL$^*$ and IP), which supports the idea of using MSBCOL to provide initial solutions to more advanced metaheuristics, such as to heuristically adapt MSBCOL$^+$ in a combinatorial local search strategy.

\subsection{Comparison with a state-of-the-art metaheuristic}\label{sec:compare}

In this subsection, we compare the solutions obtained by our newly proposed approach MSBCOL with the best ones reported in the literature using a state-of-the-art approach, namely the hybrid evolutionary algorithm of \citeA{FisPetMerCre15} (denoted henceforth as HEA). We choose to analyze in this subsection only MSBCOL rather than the complete matheuristic approach in order to establish this metaheuristic as a robust and effective method, considering that both MSBCOL and HEA can be classified as pure metaheuristics. The authors tested their algorithm on a set of small instances composed of $d$-regular graphs with up to 12 vertices and on nine large graphs from the second DIMACS implementation challenge (which is a subset of the large benchmark set described in Subsection~\ref{sec:instancias}). As the small instance set used by the authors only includes extremely small graphs and is not publicly available, we do not report results for that set.

Table~\ref{table:compare} compares the results obtained by MSBCOL and HEA for all nine large graphs tested in \citeA{FisPetMerCre15}. We remark that the results for these instances using all the approaches proposed in our paper were already presented in Subsection~\ref{sec:large}, as these nine instances represent a subset of the bencmark set described in Subsection~\ref{sec:instancias}.
The first column identifies the instance. Columns 2 to 4 report the number of vertices ($|V|$), the number of edges ($|E|$), and the solution upper bound ($m(G)$).  Columns 5 to 8 give, for MSBCOL, the best solution value ($z_M$), the average solution value for the executed number of iterations ($z_{avg}$), the running time in seconds (time(s)), and the percentual improvement over HEA ($imp_{M}$), calculated as $100\times \frac{z_{M} - z_{H}}{z_H}$ with $z_H$ being the best solution encountered by HEA. Columns 9 and 10 give, for HEA, the solution value ($z_H$), and the running time in seconds (time(s)). The best solution values are shown in bold. We point out that the authors did not report the machine used in the experiments for HEA, therefore the main goal in this subsection is to compare the quality of the solutions obtained using the two approaches. 

\begin{table}[H]
\centering
\small
\caption{Results comparing MSBCOL and HEA for a subset of the instances containing nine large graphs.} 
\begin{tabular}{lccc|cccc|cc}
\hline
\multicolumn{4}{l|}{Instance name}      & \multicolumn{4}{c|}{MSBCOL} &  \multicolumn{2}{c}{HEA~\cite{FisPetMerCre15}}        \\

& $|V|$  & $|E|$    & $m(G)$ & $z_M$ & {$z_{avg}$} &  time(s) & imp$_{M}(\%)$ & $z_H$ & time(s)\footnotemark \\
\hline
dsjc250.5.col	&	250	&	15,668	&	126	&	\textbf{51}	& {46.9} &	2.8	&	2.0	&	50	&	186.8	\\
dsjc500.1.col	&	500	&	12,458	&	59	&	\textbf{29}	& {25.6} &	2.9	&	16.0	&	25	&	113.6	\\
dsjc500.5.col	&	500	&	62,624	&	251	&	\textbf{88}	& {82.8} &	14.1	&	4.8	&	84	&	4506.9	\\
dsjr500.5.col	&	500	&	58,862	&	234	&	\textbf{183} & {177.1}	&	24.3	&	10.2	&	166	&	15017.0	\\
flat300\_28\_0.col	&	300	&	21,695	&	146	&	\textbf{57}	& {52.8} &	3.9	&	3.6	&	55	&	1048.0	\\
flat1000\_50\_0.col	&	1,000	&	245,000	&	492	&	\textbf{153} & {144.1}	&	69.7	&	10.1	&	139	&	58428.0	\\
le450\_25c.col	&	450	&	17,343	&	101	&	\textbf{48}	& {44.6} &	8.7	&	6.7	&	45	&	571.7	\\
le450\_25d.col	&	450	&	17,425	&	99	&	\textbf{48}	& {44.9} &	7.2	&	6.7	&	45	&	256.4	\\
r250.5.col	&	250	&	14,849	&	119	&	\textbf{100} & {94.9}	&	3.7	&	8.7	&	92	&	1836.2	\\
\hline
\end{tabular}
\label{table:compare}
\end{table}

 The reported values show that MSBCOL clearly outperformed HEA for all nine large instances considered in \citeA{FisPetMerCre15}. MSBCOL was able to obtain strictly better solutions for all of them, representing an 100.00\% success rate on improving over the previously best known solutions presented by HEA. The most notable performance can be seen in instances dsjr500.1.col, dsjr500.5.col and flat1000\_50\_0.col, for which MSBCOL achieved improvements over 10.00\% when compared to HEA. {Results also show that even some of the average solution values were able to improve over the previously best known solution values presented by HEA, since 4 out of 9 (44.4\%) outperform HEA's results (see dsjc500.1.col, dsjr500.5.col, flat1000\_50\_0.col, and r250.5.col).}
 
 It is noteworthy that MSBCOL was very effective when it comes to the running times, as they were below a minute for all but one instance. The maximum running time was less than 70.0 seconds for the largest instance (flat1000\_50\_0.col), which is composed of 1,000 vertices and 245,000 edges.

\footnotetext{The authors did not report the computational resources used for the experiments.}

%% file: 06_finalremarks.tex
\section{{Concluding} remarks}
\label{sec:finalremarks}

In this paper, we considered the $b$-coloring problem and proposed: the first integer programming formulation for the optimization variant of the problem, which consists in maximizing the number of colors used in a proper $b$-coloring; a {multi-start multi-greedy randomized metaheuristic}, which differs from previous {(meta)}heuristics by taking into account the structure of the problem in its mechanism; and a very effective {matheuristic approach} combining the multi-start multi-greedy randomized {meta}heuristic with a fix-and-optimize {local search procedure} using the proposed integer programming formulation. Moreover, we also proposed a benchmark set of instances to be used in future works.

Computational experiments were performed on a newly proposed benchmark set to analyze the performance of the presented techniques. The multi-greedy randomized heuristic has shown to be very effective while having very few parameters to be configured.
The integer programming formulation was able to provide satisfactory results, but it is considerably compromised as the instance size grows, considering that the number of variables have a tendency to become intractable leading to memory overflow. The fix-and-optimize local search procedure used in the matheuristic approach improved a significant amount of solutions and reported the majority of best results, demonstrating to be a very effective and promising method for the $b$-coloring and other related problems. The results have also shown that the proposed multi-start metaheuristic outperforms a state-of-the-art evolutionary algorithm for a subset of the instances, namely, all nine large instances which were considered in \citeA{FisPetMerCre15}.
Last but not least, the proposed benchmark set features a variety of instances including small and large graphs with different characteristics, which can be used in future computational experiments to verify the performance of both exact and heuristic approaches for the $b$-coloring problem.

Relevant research directions include the development of combinatorial local search approaches to overcome the memory limitations of the used large formulations. Such combinatorial local search approaches could be used together with the approaches proposed in our work in advanced metaheuristic frameworks. Another interesting possible direction is a polyhedral study of formulations for the $b$-coloring problem.